\documentclass[10pt]{article}%
\usepackage{amsmath}
\usepackage{graphicx}
\usepackage{epsfig}
\usepackage{multicol}
\usepackage{float}
\usepackage{amsfonts}
\usepackage{amssymb}
\usepackage{color}
\usepackage{xcolor}%
\newtheorem{theorem}{Theorem}[section]

\newtheorem{lemma}[theorem]{Lemma}

\newtheorem{proposition}[theorem]{Proposition}
\newtheorem{remark}{Remark}[section]

\newenvironment{proof}[1][Proof]{\textbf{#1.} }{\ \rule{0.5em}{0.5em}}

\DeclareMathOperator{\inter}{int}
\begin{document}

\title{Global solutions and asymptotic behavior for a parabolic degenerate coupled
system arising from biology}
\date{}
\author{{\normalsize Gabriela\ Li\c{t}canu* and Cristian Morales-Rodrigo**}\\{\small *Institute of Mathematics "O. Mayer", Romanian Academy,} \\{\small \ 700505 Ia\c{s}i, Romania. e-mail: litcanu@uaic.ro}\\{\small Institute of Mathematics, Witten/Herdecke University, } \\{\small D-58453 Witten, Germany.} \\{\small **Institute of Applied Mathematics and Mechanics, Warsaw University, } \\{\small 02-097 Warsaw, Poland. e-mail: cristianmatematicas@yahoo.com}}
\maketitle

\begin{abstract}
{\small In this paper we will focus on a parabolic degenerate system with
respect to unknown functions }$u$ {\small and }$w$ {\small on a bounded domain
of the two-dimensional Euclidean space. This system appears as a mathematical
model for some biological processes. Global existence and uniqueness of a
nonnegative classical H\"{o}lder continuous solution are proved. The last part
of the paper is devoted to the study of the asymptotic behavior of the
solutions.}

\ 


\ \ 


\end{abstract}

\section{Introduction}

\label{introduction} \ 

\setcounter{equation}{0} \setcounter{figure}{0}

During the last years models originated from biology earned a privileged place
in mathematical modeling and became the focus of interest of mathematicians
and biologists as well. In many cases the study of these models involves
challenging mathematical problems that originate in the intrinsic mathematical
structure of the model. Moreover the possibility of taking suitable hypotheses
is limited by the necessity to fit with experimental data of the process the
model originates in.

Let us consider the following initial-boundary problem:%
\begin{align}
&  \frac{\partial u}{\partial t}=a\Delta u-b\nabla\cdot(u\chi(w)\nabla
w)+f(u,w) & x  &  \in\Omega,\quad t\in\mathbb{R}_{+}\label{eq1_gen}\\
&  \frac{\partial w}{\partial t}=-kw^{\beta}u & x  &  \in\Omega,\quad
t\in\mathbb{R}_{+}\label{eq2_gen}\\
&  u(x,0)=u_{0}(x)\geqslant0 & x  &  \in\Omega\label{ic1_gen}\\
&  w(x,0)=w_{0}(x)>0 & x  &  \in\Omega\label{ic2_gen}%
\end{align}
where $\Omega\subseteq\mathbb{R}^{N}$ is a domain, $a,b$ and $k$ are positive
constants, $\chi(w)=w^{-\alpha}$, $0\leqslant\alpha<1$, $\beta\geqslant1$ and
$f$ is a given function. If $\Omega$ is bounded, then the system
(\ref{eq1_gen})-(\ref{ic2_gen})\ is considered together with the no-flux
boundary condition%
\begin{equation}
\frac{\partial u}{\partial\eta}-u\chi(w)\frac{\partial w}{\partial\eta
}=0\qquad\qquad\qquad x\in\partial\Omega,\quad t\in\mathbb{R}_{+}
\label{bc_gen}%
\end{equation}
where $\eta$ denotes the unit outward normal vector of $\partial\Omega$.

This system is a particular version of the well-known mathematical model
proposed by Keller and Segel \cite{keller-segel1} (see also
\cite{keller-odell1}, \cite{keller-odell2}, \cite{keller-segel2}) with an
additional reaction term $f(u,w)$ in the first equation. The Keller-Segel
model was proposed in order to describe the spatial aggregation of cellular
slime molds which move toward high concentrations of some chemical substance
secreted by the cells themselves. The function $u(x,t)$ describes the density
distribution of the cell population, $w(x,t)$ denotes the concentration of the
chemical substance at a position $x\in\Omega$ and a time $t\in\mathbb{R}_{+}$
and the function $\chi$\ is the chemotactic sensitivity.

The classical Keller-Segel model, when the second variable is also supposed to
be diffusive, has been subject of many papers (see, for example, the surveys
of Horstmann \cite{horstmann1}, \cite{horstmann2} and the references given
therein). In the literature there are many theoretical results for the
Keller-Segel model concerning existence and uniqueness as well as the
qualitative behavior of the solutions. Most of the results were focused on the
global existence of solutions versus blow-up in finite time. Both behaviors
strongly depend on the initial data and space dimension.

The system (\ref{eq1_gen})-(\ref{bc_gen}) also appears as a simplified
mathematical model describing the tumor growth when the formation of new blood
vessels from the pre-existing vascular network is initiated (angiogenesis). In
this case, the function $u(x,t)$ describes the tumor cells density and
$w(x,t)$ denotes the density of the extracellular matrix (the surrounding
healthy tissue degraded locally by the action of tumor cells). There are
several models of different stages of the angiogenesis process incorporating
also the action of some degradative enzymes, cell cycle elements or cell age
structures. For a more thorough biological background and numerical results
concerning the angiogenesis process see, for example, \cite{anderson1},
\cite{anderson2}, \cite{chaplain}, \cite{RewAngio}, \cite{sheratt1}. We refer
also to \cite{walker1}, \cite{walker2}\ where the global existence and
uniqueness of solutions in the case of some systems related with this process
are investigated.

Previously, a version of the system (\ref{eq1_gen})-(\ref{bc_gen}) was studied
by Rascle in \cite{rascle2} (see also \cite{rascle1}) with the boundary
condition (\ref{bc_gen}) replaced with%
\begin{equation}
\frac{\partial u}{\partial\eta}=0\text{.} \label{bcR}%
\end{equation}
Instead of (\ref{ic2_gen}) he takes a positive constant as initial condition
for the function $w$, $w_{0}(x)\equiv w_{0}>0$ and $f(u,v)$ satisfying the
following condition%
\begin{equation}
\exists L>0,\ \forall u\in\mathbb{R},\ \forall w>0,\ \left\vert
f(u,w)\right\vert \leqslant L\left\vert u\right\vert \text{.} \label{f}%
\end{equation}

In the previous hypotheses, the local existence and uniqueness of a classical
H\"{o}lder continuous solution of the system (\ref{eq1_gen})-(\ref{ic2_gen})
has been proved when $\Omega\subset\mathbb{R}^{N}$ is a bounded domain with
smooth boundary $\partial\Omega$. The global existence has been shown in one
space dimension. We mention that another result, in the one dimensional space,
concerning the global existence and uniqueness of classical solutions for a
similar system is given in \cite{gua}.

In more than one dimension, when $\Omega$ is the whole space $\mathbb{R}^{N}$,
the system (\ref{eq1_gen})-(\ref{bc_gen}) has been considered in
\cite{perthame1}, \cite{perthame2}, \cite{perthame3} with $\chi(w)$ a given
positive function on $\mathbb{R}_{+}$\ such that $w\chi(w)$ is strictly
increasing (thus including the case $\chi(w)=w^{-\alpha}$, $0\leqslant
\alpha<1$) and $f\equiv0$. In this case the global existence of weak solutions
has been proved.

In \cite{guar} the authors considered the problem (\ref{eq1_gen}%
)-(\ref{ic2_gen}) in a more general form under Dirichlet conditions. Assuming
that a priori $L^{\infty}$ bounds are available they proved the local and
global existence of weak solutions.

Finally, we cite here the paper \cite{tello} where the author considers
instead of the equation (\ref{eq2_gen}) the following one%
\[
\frac{\partial w}{\partial t}=g(u,w)
\]
but under some hypotheses on $g$ that are not satisfied in the case we shall
consider in this paper (see also \cite{fontelos}, \cite{friedman}).

Our aim in this paper is to prove the global existence in time and uniqueness
of a classical H\"{o}lder continuous solution\ for the problem (\ref{eq1_gen}%
)-(\ref{ic2_gen}) when $\Omega\subset\mathbb{R}^{2}$ is a bounded domain with
smooth boundary $\partial\Omega$ and the reaction term is the logistic growing
function. Also the long time asymptotic behavior of the solution is
investigated. In order to simplify the presentation of the results, we shall
consider in what follows the case $\alpha=0$. The more general cases
$\beta\geqslant1$, $0\leqslant\alpha<1$\ (or even when the function $\chi$ is
a more general decreasing function) can be treated similarly, the estimations
being more tedious.

This paper is organized as follows. In Section \ref{notations} we review some
basic facts concerning the notations and terminology used through the paper
and we also give some auxiliary results. The proof of the local existence in
time and uniqueness of a classical solution is accomplished by applying a
fixed point argument in a suitably chosen function space and is presented in
Section \ref{local}.

In Section \ref{global_existence} we will be concerned with the global
existence in time of the classical solutions and for this we will begin by
establishing a priori bounds.

In Subsection \ref{lyap} we obtain a Lyapunov function for the system
(independent of the space dimension) by an analogous method as in
\cite{perthame2} (see also \cite{perthame3}, \cite{gaj}, \cite{horstmann0}).
We derive $L^{p}(\Omega)$ estimates independent on time in Subsection
\ref{estimat2}. After establishing a priori $L^{\infty}\left(  \Omega\right)
$ uniform bounds in Subsection \ref{estimat3}, we proceed to prove the
existence of global H\"{o}lder continuous solutions imposing that the initial
data are smooth enough.

Section \ref{asymptotic} is devoted to the study of the long time asymptotic
behavior of the solutions. More precisely, we prove that the solution
converges to a steady state of the system, exponentially if $\beta=1$ and at a
polynomial rate if $\beta>1$.

\section{Preliminaries}

\label{notations} \ 

\setcounter{equation}{0} \setcounter{figure}{0}

Hereafter we assume that $\Omega\subset\mathbb{R}^{N}$, $N\geqslant1$ is a
bounded domain with smooth boundary $\partial\Omega\in C^{l+2}\left(
\mathbb{R}^{N-1}\right)  $. Given $T\in(0,+\infty]$, we consider the
cylindrical domain denoted by $\Omega_{T}=\Omega\times(0,T)$ with lateral
surface $\partial\Omega_{T}=\partial\Omega\times(0,T)$.

We are using in this paper the standard notation of function spaces.
$L^{p}(\Omega)$ and $W^{m,p}(\Omega)$ with $1\leqslant p\leqslant\infty$,
$m\geqslant1$ are the Lebesgue spaces and respectively, Sobolev spaces of
functions on $\Omega$. For a general Banach space $X$, its norm is denoted by
$\left\Vert \cdot\right\Vert _{X}$. The space $L^{p}(0,T;X)$ is the Banach
space of all Bochner measurable functions $f:(0,T)\rightarrow X$ such that
$\left\Vert f\right\Vert _{X}\in L^{p}(0,T)$.

For a positive integer $n$ we consider the Banach space $W_{p}^{2n,n}%
(\Omega_{T})=\left\{  f;\text{ }D_{t}^{r}D_{x}^{s}f\in L^{p}(\Omega
_{T}),\text{ }2r+\left\vert s\right\vert \leqslant2n\right\}  $ together with
the norm
\[
\left\Vert f\right\Vert _{p,\Omega_{T}}^{\left(  2n\right)  }=\sum
_{0\leqslant2r+\left\vert s\right\vert \leqslant2n}\left\Vert D_{t}^{r}%
D_{x}^{s}f\right\Vert _{L^{p}(\Omega_{T})}.
\]

Given a non-integer positive number $0<l<1$, we denote by $C^{l+i,l/2+i/2}%
(\overline{\Omega}_{T})$, $i=1,2$ the H\"{o}lder space of exponents $l+i$ and
$l/2+i/2$ by respect to $x$, respectively $t$ of continuous and bounded
functions $\left\{  f(x,t)\right\}  $ defined on $\overline{\Omega}_{T}$,
provided with continuous and bounded derivatives $\left\{  D_{t}^{r}D_{x}%
^{s}f(x,t)\right\}  $ for $2r+\left\vert s\right\vert \leqslant i$. It is
endowed with the norm given by%
\begin{align*}
\left\vert f\right\vert _{\Omega_{T}}^{(l+i)}  &  :=\sum_{2r+\left\vert
s\right\vert =i}\left\langle D_{t}^{r}D_{x}^{s}f\right\rangle _{x,\Omega_{T}%
}^{(l)}+\sum_{\max\left\{  0,i-1\right\}  \leqslant2r+\left\vert s\right\vert
\leqslant i}\left\langle D_{t}^{r}D_{x}^{s}f\right\rangle _{t,\Omega_{T}%
}^{\left(  (l-2r-\left\vert s\right\vert )/2+i/2\right)  }+\\
&  +\sum_{0\leqslant2r+\left\vert s\right\vert \leqslant i}\max
\limits_{(x,t)\in\overline{\Omega}_{T}}\left\vert D_{t}^{r}D_{x}%
^{s}f\right\vert
\end{align*}
where
\[
\left\langle f\right\rangle _{x,\Omega_{T}}^{(l)}:=\sup
\limits_{\substack{(x,t),(x^{\prime},t)\in\overline{\Omega}_{T}\\\left\vert
x-x^{\prime}\right\vert \leqslant\rho_{0}}}\frac{\left\vert f(x,t)-f(x^{\prime
},t)\right\vert }{\left\vert x-x^{\prime}\right\vert ^{l}},\qquad\left\langle
f\right\rangle _{t,\Omega_{T}}^{(l/2)}:=\,\sup
\limits_{\substack{(x,t),(x,t^{\prime})\in\overline{\Omega}_{T}\\\left\vert
t-t^{\prime}\right\vert \leqslant\rho_{0}}}\frac{\left\vert
f(x,t)-f(x,t^{\prime})\right\vert }{\left\vert t-t^{\prime}\right\vert ^{l/2}%
}\,.
\]

This norm mentioned above depends on $\rho_{0}>0$, but changing this constant
leads to an equivalent norm.

Throughout this paper we denote by $C$, $C_{i}$ ($i=1,2,...$) positive
constants which are independent of time, but we shall indicate explicitly on
which other parameters they are dependent, if it will be the case. The
constants $C$ are not necessarily the same at different occurrences.

Some properties for the norms in the H\"{o}lder spaces which will be used
often in the next sections are given below. Since the proofs are standard, but
tedious, we omit the details.

\begin{lemma}
\label{lemma1}If $f(x,t)\in C^{l+2,l/2+1}(\overline{\Omega}_{T})$, $0<l<1$,
then\ we have:

\begin{itemize}
\item[(i)] $\frac{\partial f}{\partial t}\in C^{l,l/2}(\overline{\Omega}_{T})$,

\item[(ii)] $\frac{\partial f}{\partial x_{j}}\in C^{l+1,l/2+1/2}%
(\overline{\Omega}_{T})$, $j=1,...,N$,

\item[(iii)] $\Delta f\in C^{l,l/2}(\overline{\Omega}_{T})$,

\item[(iv)] $\frac{\partial f}{\partial\eta}\in C^{l+1,l/2+1/2}(\overline
{\Omega}_{T})$, where $\eta$ denotes the unit outward normal vector of
$\partial\Omega$.
\end{itemize}
\end{lemma}

\begin{lemma}
\label{lemma6}If $f(x,t)\in C^{l+2,l/2+1}(\overline{\Omega}_{T})$, $0<l<1$,
then $H(x,t)=\int_{0}^{t}f(x,s)\,ds\in C^{l+2,l/2+1}(\overline{\Omega}_{T})$.
Moreover,%
\begin{equation}
\left\vert H\right\vert _{\Omega_{T}}^{(l+2)}\leqslant C\max\left\{
T,T^{\left(  1-l\right)  /2}\right\}  \left\vert f(x,t)\right\vert
_{\Omega_{T}}^{(l+2)}+\left\vert f(x,0)\right\vert _{\Omega_{T}}^{(l)}\text{.}
\label{Fintegral}%
\end{equation}

\end{lemma}

\begin{lemma}
\label{lemma11}If $f,g\in C^{l+i,l/2+i/2}(\overline{\Omega}_{T})$, $0<l<1$,
then $fg\in C^{l+i,l/2+i/2}(\overline{\Omega}_{T})$ and
\begin{equation}
\left\vert fg\right\vert _{\Omega_{T}}^{(l+i)}\leqslant C\left\vert
f\right\vert _{\Omega_{T}}^{(l+i)}\left\vert g\right\vert _{\Omega_{T}%
}^{(l+i)} \label{prod2}%
\end{equation}
for $i=0,1,2$.
\end{lemma}

\begin{lemma}
\label{lemmaR}(\cite{rascle2}, Lemma 1) Let $\varphi,\psi:\Omega
_{T}\rightarrow K\subset\mathbb{R}^{N}$, where $K$ is a compact in
$\mathbb{R}^{N}$, be two functions in $\left(  C^{l+2,l/2+1}(\overline{\Omega
}_{T})\right)  ^{N}$ and let $f\in C^{3}(K)$. Then $f\circ\varphi$ and
$f\circ\psi$ are in $C^{l+2,l/2+1}(\overline{\Omega}_{T})$\ and we have%
\begin{equation}
\left\vert f\circ\varphi-f\circ\psi\right\vert _{\Omega_{T}}^{(l+2)}%
\leqslant\Phi\left\Vert f\right\Vert _{C^{3}(K)}\left(  \left\vert
\varphi-\psi\right\vert _{\Omega_{T}}^{(l+2)}\right)  ^{\gamma} \label{R}%
\end{equation}
where $\gamma=\min\left\{  l/2,1-l\right\}  $ and $\Phi=\Phi(\left\vert
\varphi\right\vert _{\Omega_{T}}^{(l+2)},\left\vert \psi\right\vert
_{\Omega_{T}}^{(l+2)})$ is an increasing function on both its arguments.
\end{lemma}

The remaining of this section is devoted to some general results for the
existence of solutions for parabolic equations. We consider the problem:%
\begin{align}
&  \frac{\partial u}{\partial t}-\Delta u+\sum\limits_{i=1}^{N}a_{i}%
(x,t)\frac{\partial u}{\partial x_{i}}+a(x,t)u=F(x,t) &  &  \left(
x,t\right)  \in\Omega_{T}\label{Lu}\\
&  \frac{\partial u}{\partial\eta}-u\frac{\partial g}{\partial\eta
}(x,t)=G(x,t) &  &  \left(  x,t\right)  \in\partial\Omega_{T}\label{Lubc}\\
&  u(x,0)=u_{0}(x) &  &  x \in\Omega. \label{Lu0}%
\end{align}
Let us remark that, if we make the change of variables
$v(x,t)=u(x,t)e^{-g(x,t)}$ the system (\ref{Lu})-(\ref{Lu0}) becomes:%
\begin{align}
&  \frac{\partial v}{\partial t}-\Delta v+\sum\limits_{i=1}^{N}b_{i}%
(x,t)\frac{\partial v}{\partial x_{i}}+b(x,t)v=\widetilde{F}(x,t) &  &
\left(  x,t\right)  \in\Omega_{T}\label{Lv}\\
&  \frac{\partial v}{\partial\eta}=\widetilde{G}(x,t) &  &  \left(
x,t\right)  \in\partial\Omega_{T}\label{Lvbc}\\
&  v(x,0)=v_{0}(x) &  &  x \in\Omega\label{Lv0}%
\end{align}
where the coefficients are given by:%
\begin{align}
&  b_{i}(x,t)=a_{i}(x,t)-2\frac{\partial g}{\partial x_{i}}(x,t),\qquad
1\leqslant i\leqslant N\label{bx1}\\
&  b(x,t)=a(x,t)+\frac{\partial g}{\partial t}(x,t)-\Delta g+\sum
\limits_{i=1}^{N}a_{i}(x,t)\frac{\partial g}{\partial x_{i}}(x,t)-\sum
\limits_{i=1}^{N}\left(  \frac{\partial g}{\partial x_{i}}(x,t)\right)
^{2}\label{bx2}\\
&  \widetilde{F}(x,t)=F(x,t)e^{-g(x,t)}\label{F}\\
&  \widetilde{G}(x,t)=G(x,t)e^{-g(x,t)}\label{G}\\
&  v_{0}(x)=u_{0}(x)e^{-g(x,0)}. \label{vx}%
\end{align}

\begin{theorem}
\label{rascle}(\cite{rascle1}, Theorem II.2) Let $0<l<1$ and $\Omega
\subset\mathbb{R}^{N}$ be a bounded domain with the boundary $\partial
\Omega\in C^{l+2}$ and $0<T<\infty$. We suppose that the following hypotheses
are satisfied:

\begin{itemize}
\item the coefficients $b_{i}(x,t)$ ($1\leqslant i\leqslant N$), $b(x,t)$
belong to the space $C^{l,l/2}(\overline{\Omega}_{T})$;

\item $\widetilde{F}(x,t)\in C^{l,l/2}(\overline{\Omega}_{T})$, $\widetilde
{G}(x,t)\in C^{l+1,l/2+1/2}(\overline{\partial\Omega}_{T})$ and $v_{0}(x)\in
C^{l+2}(\overline{\Omega})$;

\item the compatibility condition $\frac{\partial v}{\partial\eta
}(x,0)=\widetilde{G}(x,0)$ is satisfied for every $x\in\partial\Omega$.
\end{itemize}

Then the problem (\ref{Lv})-(\ref{Lv0}) has a unique solution $v(x,t)\in
C^{l+2,l/2+1}(\overline{\Omega}_{T})$ which verifies%
\begin{equation}
\left\vert v\right\vert _{\Omega_{T}}^{(l+2)}\leqslant\Theta\left(  \left\vert
\widetilde{F}\right\vert _{\Omega_{T}}^{(l)}+\left\vert \widetilde
{G}\right\vert _{\partial\Omega_{T}}^{(l+1)}+\left\vert v_{0}\right\vert
_{\Omega}^{(l+2)}\right)  \label{depend_lady}%
\end{equation}
where $\Theta=\Theta(T,\mu(T))$ is an increasing function on $T$ and on the
quantity
\[
\mu(T)=\sum_{i=1}^{N}\left\vert b_{i}(x,t)\right\vert _{\Omega_{T}}%
^{(l)}+\left\vert b(x,t)\right\vert _{\Omega_{T}}^{(l)}\text{.}%
\]

\end{theorem}

\begin{theorem}
\label{lady}Let $0<l<1$ and $\Omega\subset\mathbb{R}^{N}$ be a bounded domain
with the boundary $\partial\Omega\in C^{l+2}$ and $0<T<\infty$. We suppose
that the following hypotheses are satisfied:

\begin{itemize}
\item the coefficients $a_{i}(x,t)$ ($1\leqslant i\leqslant N$), $a(x,t)$
belong to the space $C^{l,l/2}(\overline{\Omega}_{T})$;

\item $F(x,t)\in C^{l,l/2}(\overline{\Omega}_{T})$, $G(x,t)\in C^{l+1,l/2+1/2}%
(\overline{\partial\Omega}_{T})$, $g(x,t)\in C^{l+2,l/2+1}(\overline
{\partial\Omega}_{T})$ and $u_{0}(x)\in C^{l+2}(\overline{\Omega})$;

\item the compatibility condition $\frac{\partial u_{0}}{\partial\eta}%
-u_{0}\frac{\partial g}{\partial\eta}(x,0)=G(x,0)$ is satisfied for every
$x\in\partial\Omega$.
\end{itemize}

Then the problem (\ref{Lu})-(\ref{Lu0}) has a unique solution $u(x,t)\in
C^{l+2,l/2+1}(\overline{\Omega}_{T})$ which verifies%
\begin{equation}
\left\vert u\right\vert _{\Omega_{T}}^{(l+2)}\leqslant\Psi\left(  \left\vert
F\right\vert _{\Omega_{T}}^{(l)}+\left\vert G\right\vert _{\partial\Omega_{T}%
}^{(l+1)}+\left\vert u_{0}\right\vert _{\Omega}^{(l+2)}\right)
\label{depend_u}%
\end{equation}
where $\Psi=\Psi\left(  T,\left\vert g\right\vert _{\Omega_{T}}^{(l+2)}%
,\mu(T)\right)  $ is an increasing function in $T,$ in $\left\vert
g\right\vert _{\Omega_{T}}^{(l+2)}$ and in the quantity
\[
\mu(T)=\sum_{i=1}^{N}\left\vert b_{i}(x,t)\right\vert _{\Omega_{T}}%
^{(l)}+\left\vert b(x,t)\right\vert _{\Omega_{T}}^{(l)}%
\]
where $b_{i}(x,t)$ ($1\leqslant i\leqslant N$), $b(x,t)$ are given by
(\ref{bx1}), (\ref{bx2}).
\end{theorem}

\begin{proof}
The existence and the uniqueness of the solution is proved in \cite{lady},
Chapter IV, Theorem 5.3. The only thing that we want to point out is the
increasing dependence of the function $\Psi$\ on its arguments.

From Lemma \ref{lemma11} we obtain%
\[
\left\vert u(x,t)\right\vert _{\Omega_{T}}^{(l+2)}=\left\vert v(x,t)e^{g(x,t)}%
\right\vert _{\Omega_{T}}^{(l+2)}\leqslant C\left\vert v(x,t)\right\vert
_{\Omega_{T}}^{(l+2)}\left\vert e^{g(x,t)}\right\vert _{\Omega_{T}}^{(l+2)}.
\]
Now, taking into account (\ref{depend_lady}) and Lemma \ref{lemmaR}, we obtain
immediately the relation (\ref{depend_u}).
\end{proof}

\section{Local existence in time and uniqueness of classical solutions}

\label{local} \ 

\setcounter{equation}{0} \setcounter{figure}{0}

As we have already mentioned in Introduction, in order to simplify the
presentation of the results, we consider the system (\ref{eq1_gen}%
)-(\ref{bc_gen}) when $\alpha=0$. We consider, without loss of generality, the
normalized system, which means $a=b=k=1$, with the growing source term, more
precisely%
\begin{align}
&  \frac{\partial u}{\partial t}=\Delta u-\nabla\cdot(u\nabla w)+\delta u(1-u)
& x  &  \in\Omega,\quad t\in\mathbb{R}_{+}\label{eq1}\\
&  \frac{\partial w}{\partial t}=-w^{\beta}u & x  &  \in\Omega,\quad
t\in\mathbb{R}_{+}\label{eq2}\\
&  \frac{\partial u}{\partial\eta}-u\frac{\partial w}{\partial\eta}=0 & x  &
\in\partial\Omega,\quad t\in\mathbb{R}_{+}\label{bc}\\
&  u(x,0)=u_{0}(x)\geqslant0 & x  &  \in\Omega\label{ic1}\\
&  w(x,0)=w_{0}(x)>0 & x  &  \in\Omega\label{ic2}%
\end{align}
where $\delta\geqslant0$ and $\beta\geqslant1$.

\begin{remark}
In what follows the computations are made for $\beta>1$. The same results are
true also for $\beta=1$, but the estimates will be different. We have
considered the growing source term, nevertheless the results are valid also in
the case of more general functions satisfying appropriate conditions.
\end{remark}

The arguments given in this Section are similar to those of Rascle
\cite{rascle1}, \cite{rascle2}. Because in our case the boundary condition is
different and the function $f$ does not satisfy the condition (\ref{f}), we
briefly give the proof for the local existence for the sake of completeness.

Let us remark that, if $\beta>1$, we can rewrite the initial problem
(\ref{eq1})-(\ref{ic2}):%
\begin{align}
&  \frac{\partial u}{\partial t}=\nabla\left(  \nabla u-u\cdot\nabla\left[
\left(  w_{0}^{1-\beta}+U\right)  ^{\frac{1}{1-\beta}}\right]  \right)
+\delta u\left(  1-\frac{1}{\beta-1}\frac{\partial U}{\partial t}\right)  & x
&  \in\Omega,\quad t\in\mathbb{R}_{+}\label{u1}\\
&  \frac{\partial u}{\partial\eta}=u\frac{\partial}{\partial\eta}\left[
\left(  w_{0}^{1-\beta}+U\right)  ^{\frac{1}{1-\beta}}\right]  & x  &
\in\partial\Omega,\quad t\in\mathbb{R}_{+}\label{u2}\\
&  u(x,0) =u_{0}(x) & x  &  \in\Omega\label{u3}\\
&  U=\left(  \beta-1\right)  \int\limits_{0}^{t}u(x,s)ds & x  &  \in
\Omega,\quad t\in\mathbb{R}_{+}\label{U1}\\
&  U(x,0)=0. & x  &  \in\Omega\label{U2}%
\end{align}

We consider now the following linear problem in the variable $u$%
\begin{align}
&  \frac{\partial u}{\partial t}=\Delta u-\sum\limits_{i=1}^{N}a_{i}%
(x,t)\frac{\partial u}{\partial x_{i}}-a(x,t)u & x  &  \in\Omega,\quad
t\in\mathbb{R}_{+}\label{p1}\\
&  \frac{\partial u}{\partial\eta}=u\frac{\partial g}{\partial\eta} & x  &
\in\partial\Omega,\quad t\in\mathbb{R}_{+}\label{p2}\\
&  u(x,0)=u_{0}(x)\geqslant0 & x  &  \in\Omega\label{p3}%
\end{align}
where the coefficients are given by%
\begin{equation}
g(x,t)=\left(  w_{0}^{1-\beta}(x)+\phi(x,t)\right)  ^{\frac{1}{1-\beta}%
},\qquad a_{i}(x,t)=\frac{\partial g}{\partial x_{i}},\qquad a(x,t)=\Delta
g-\delta\left(  1+g^{-\beta}\frac{\partial g}{\partial t}\right)  .
\label{coef_g}%
\end{equation}

\begin{theorem}
\label{local existence}Let $0<l<1$, $\Omega\subset\mathbb{R}^{N}$ be a domain
with $C^{l+2}$ boundary $\partial\Omega$ and $0<T<\infty$. We suppose that the
following hypotheses are satisfied:

\begin{itemize}
\item $\phi\in C^{l+2,l/2+1}(\overline{\Omega}_{T})$, $w_{0}\in C^{l+2}%
(\overline{\Omega})$, $u_{0}\in C^{l+2}(\overline{\Omega})$;

\item the compatibility condition $\frac{\partial u_{0}}{\partial\eta
}(x)=u_{0}(x)\frac{\partial g}{\partial\eta}(x,0)$ is satisfied for every
$x\in\partial\Omega$.
\end{itemize}

Then the problem (\ref{p1})-(\ref{p3}) has a unique nonnegative solution
$u(x,t)\in C^{l+2,l/2+1}(\overline{\Omega}_{T})$ which verifies%
\begin{equation}
\left\vert u\right\vert _{\Omega_{T}}^{(l+2)}\leqslant\Psi\left\vert
u_{0}\right\vert _{\Omega}^{(l+2)} \label{local_u}%
\end{equation}
where $\Psi=\Psi\left(  T,\left\vert g\right\vert _{\Omega_{T}}^{(l+2)}%
,\mu(T)\right)  $ is an increasing function in $T,$ in $\left\vert
g\right\vert _{\Omega_{T}}^{(l+2)}$ and in the quantity%
\begin{equation}
\mu(T)=\sum\limits_{i=1}^{N}\left\vert a_{i}\right\vert _{\Omega_{T}}%
^{(l)}+\left\vert \frac{\partial g}{\partial t}-\delta\left(  1+g^{-\beta
}\frac{\partial g}{\partial t}\right)  \right\vert _{\Omega_{T}}^{(l)}.
\label{th}%
\end{equation}

\end{theorem}

\begin{proof}
Taking into account the properties of the norm in H\"{o}lder spaces (see Lemma
\ref{lemma1}, Lemma \ref{lemmaR}), we have%
\[
a_{i}(x,t)\in C^{l,l/2}(\overline{\Omega}_{T}),\quad a(x,t)\in C^{l,l/2}%
(\overline{\Omega}_{T}),\quad g(x,t)\in C^{l+2,l/2+1}(\overline{\Omega}%
_{T}),\qquad i=1,...,N
\]
so by Theorem \ref{lady} we obtain that the problem (\ref{p1})-(\ref{p3}) has
a unique solution $u(x,t)\in C^{l+2,l/2+1}(\overline{\Omega}_{T})$. Moreover,
taking into account (\ref{coef_g}), this solution verifies (\ref{local_u}).

The nonnegativity of the solution follows from the maximum principle.
\end{proof}

We shall prove now the local existence of the solution for the problem
(3.1)-(3.5) using a fixed point argument. We consider the set%
\begin{equation}
X(T,\sigma)=\left\{  \phi\in C^{l+2,l/2+1}(\overline{\Omega_{T}}%
);\quad\left\vert \phi\right\vert _{\Omega_{T}}^{(l+2)}\leqslant\sigma
,\quad\phi(.,0)=0,\text{\quad}\phi\geqslant0\text{ and }\frac{\partial\phi
}{\partial t}\geqslant0\text{ in }\Omega_{T}\right\}  \label{X}%
\end{equation}
where $\sigma$ is a positive constant. We define the following operators%
\[
S:X\rightarrow C^{l+2,l/2+1}(\overline{\Omega}_{T}),\qquad S(\phi)=u,
\]
where $u$ is the unique solution of the problem (\ref{p1})-(\ref{p3}), and%
\[
R:C^{l+2,l/2+1}(\overline{\Omega_{T}})\rightarrow C^{l+2,l/2+1}(\overline
{\Omega}_{T}),\qquad R(u)=U,
\]
where $U$ is given by the relation (\ref{U1}).

Let us observe that, in order to find a solution of the problem (\ref{u1}%
)-(\ref{U2}), it is enough to find a fixed point for the application
\[
R\circ S:X\rightarrow C^{l+2,l/2+1}(\overline{\Omega}_{T}),\qquad\left(
R\circ S\right)  (\phi)=R(u)=U.
\]

\begin{theorem}
\label{contraction}Let $0<l<1$, $\Omega\subset\mathbb{R}^{N}$ be a domain with
$C^{l+2}$ boundary $\partial\Omega$. We assume that the hypotheses of Theorem
\ref{local existence}\ are satisfied and, moreover, we suppose that
$\left\vert u_{0}\right\vert _{\Omega}^{(l)}<\sigma/\left(  2\left(
\beta-1\right)  \right)  $. Then for every $\zeta>0$ there exists $T_{1}>0$
such that, for all $\tau\in(0,T_{1}]$ the following properties are true:

\begin{itemize}
\item[(i)] the closed convex set $X(\tau,\sigma)$ is invariant by $R\circ S$;

\item[(ii)] the operator $R\circ S$ satisfies the following inequality in
$X(\tau,\sigma)$ with respect to the norm $\left\vert \cdot\right\vert
_{\Omega_{\tau}}^{(l+2)}$:%
\begin{equation}
\left\vert \left(  R\circ S\right)  (\phi)-\left(  R\circ S\right)
(\psi)\right\vert _{\Omega_{\tau}}^{(l+2)}\leqslant\zeta\left(  \left\vert
\phi-\psi\right\vert _{\Omega_{\tau}}^{(l+2)}\right)  ^{\gamma}
\label{U_ineqcontr}%
\end{equation}
where $\gamma=\min\left\{  l/2,1-l\right\}  $. Therefore, $R\circ S$ has a
unique fixed point $\phi$ in $X(\tau,\sigma)$.
\end{itemize}
\end{theorem}

\begin{proof}
$(i)$ Because $u(x,t)$ is the unique solution of the problem (\ref{p1}%
)-(\ref{p3}) and taking into account Lemma \ref{lemma6}, Theorem
\ref{local existence} and the relation (\ref{local_u}), for every
$0<\tau\leqslant T$, we obtain%
\begin{align}
\left\vert \left(  R\circ S\right)  (\phi)\right\vert _{\Omega_{\tau}%
}^{(l+2)}  &  \leqslant C\left(  \beta-1\right)  \max\left\{  \tau^{\left(
1-l\right)  /2},\tau\right\}  \Psi\left(  \tau,\left\vert g(x,t)\right\vert
_{\Omega_{\tau}}^{(l+2)},\mu(\tau)\right)  \left\vert u_{0}\right\vert
_{\Omega}^{(l+2)}+\nonumber\\
&  +\left(  \beta-1\right)  \left\vert u_{0}\right\vert _{\Omega_{\tau}}%
^{(l)}. \label{U_second}%
\end{align}
where $C$ is a constant independent on $\tau$. Now, in order to estimate the
function $\Psi\left(  \tau,\left\vert g\right\vert _{\Omega_{\tau}}%
^{(l+2)},\mu(\tau)\right)  $ which appears in (\ref{U_second}), we estimate
first the norm $\left\vert g\right\vert _{\Omega_{\tau}}^{(l+2)}$ using Lemma
\ref{lemmaR}
\begin{equation}
\left\vert g\right\vert _{\Omega_{\tau}}^{(l+2)}\leqslant C_{1}\left(
\left\vert w_{0}^{1-\beta}(x)\right\vert _{\Omega_{\tau}}^{(l+2)}%
+\sigma\right)  ^{\gamma} \label{g}%
\end{equation}
where $C_{1}=C_{1}(\left\vert w_{0}(x)\right\vert _{\Omega_{\tau}}%
^{(l+2)},\sigma)$ and $\gamma=\min\left\{  l/2,1-l\right\}  $. Taking into
account Lemma \ref{lemma1} and (\ref{g}) we obtain
\begin{equation}
\mu(\tau)\leqslant C_{2}+\frac{\delta}{\beta-1}\sigma+\delta\label{iu}%
\end{equation}
where $C_{2}=C_{2}\left(  \left\vert w_{0}(x)\right\vert _{\Omega_{\tau}%
}^{(l+2)},\sigma\right)  $.

From Theorem \ref{local existence} we know that the function $\Psi$ is
increasing on $\tau$, $\left\vert g\right\vert _{\Omega_{\tau}}^{(l+2)}$ and
$\mu(\tau)$, so we obtain\ from (\ref{g}) and (\ref{iu}) for $0<\tau\leqslant
T$
\begin{equation}
\Psi\left(  \tau,\left\vert g(x,t)\right\vert _{\Omega_{\tau}}^{(l+2)}%
,\mu(\tau)\right)  \leqslant\Psi\left(  \tau,C_{1}\left(  \left\vert
w_{0}^{1-\beta}(x)\right\vert _{\Omega_{\tau}}^{(l+2)}+\sigma\right)
^{\gamma},C_{2}+\frac{\delta}{\beta-1}\sigma+\delta\right)  =:\Lambda(\sigma).
\label{second}%
\end{equation}
Finally, from (\ref{U_second}), we obtain
\[
\left\vert \left(  R\circ S\right)  (\phi)\right\vert _{\Omega_{\tau}}%
^{(l+2)}<C\left(  \beta-1\right)  \max\left\{  \tau^{\left(  1-l\right)
/2},\tau\right\}  \Lambda(\sigma)\left\vert u_{0}\right\vert _{\Omega}%
^{(l+2)}+\frac{1}{2}\sigma.
\]
It follows that for $\tau>0$ sufficiently small $X(\tau,\sigma)$ is invariant
by $R\circ S$. Let $T_{2}>0$ be sufficiently small, such that, for all
$0<\tau\leqslant T_{2}$, $X(\tau,\sigma)$ is invariant by $R\circ S$.

$(ii)$ Let $\phi,\overline{\phi}\in X(T_{2},\sigma)$ and%
\[
U=R\left(  u\right)  =\left(  R\circ S\right)  \left(  \phi\right)
,\qquad\overline{U}=R\left(  \overline{u}\right)  =\left(  R\circ S\right)
\left(  \overline{\phi}\right)  .
\]
It is easy to see that the function $z=u-\overline{u}$ satisfies the problem%
\begin{align*}
&  \frac{\partial z}{\partial t}=\Delta z-\sum\limits_{i=1}^{N}a_{i}%
(x,t)\frac{\partial z}{\partial x_{i}}-a(x,t)z+\overline{F}(x,t) & x  &
\in\Omega,\quad t\in\mathbb{R}_{+}\\
&  \frac{\partial z}{\partial\eta}=z\frac{\partial g}{\partial\eta}%
+\overline{G}(x,t) & x  &  \in\partial\Omega,\quad t\in\mathbb{R}_{+}\\
&  z(x,0)=0 & x  &  \in\Omega
\end{align*}
where $g(x,t)$ is given by (\ref{coef_g}), $\overline{g}(x,t)=\left(
w_{0}^{1-\beta}(x)+\overline{\phi}(x,t)\right)  ^{\frac{1}{1-\beta}}$ and
\[
\overline{F}(x,t)=\nabla\left(  \overline{u}\cdot\nabla\left(  g-\overline
{g}\right)  \right)  +\frac{\delta}{\beta-1}\overline{u}\frac{\partial
}{\partial t}\left(  \phi-\overline{\phi}\right)  ,\qquad\overline
{G}(x,t)=\overline{u}\frac{\partial}{\partial\eta}\left[  g-\overline
{g}\right]  .
\]
Let us notice that $\overline{G}(x,0)=0$, so the function $z(x,t)=\left(
u-\overline{u}\right)  (x,t)$ satisfies the compatibility condition
$\frac{\partial z}{\partial\eta}(x,0)-z(x,0)\frac{\partial g}{\partial\eta
}=G(x,0)$. We obtain, taking into account Theorem \ref{lady}%
\[
\left\vert \left(  R\circ S\right)  (\phi)-\left(  R\circ S\right)
(\overline{\phi})\right\vert _{\Omega_{\tau}}^{(l+2)}\leqslant C_{3}\left(
\beta-1\right)  \max\left\{  \tau^{\left(  1-l\right)  /2},\tau\right\}
\Psi\left(  \sigma\right)  \left\vert u_{0}\right\vert _{\Omega_{\tau}%
}^{(l+2)}\left(  \left\vert \phi-\overline{\phi}\right\vert _{\Omega_{\tau}%
}^{(l+2)}\right)  ^{\gamma}%
\]
where $\gamma=\min\left\{  l/2,1-l\right\}  $ and $C_{3}=C_{3}(\sigma)$. By
taking $\tau$ sufficiently small the inequality (\ref{U_ineqcontr}) follows.
We choose now $T_{1}<T_{2}$ such that $(i)$ and $(ii)$ are fulfilled for all
$\tau\in(0,T_{1}]$.

We define now the following two sequences%
\[
u_{n}=S(U_{n}),\qquad U_{n+1}=R(u_{n})=(R\circ S)(U_{n})
\]
where $U_{0}=0$. It follows from the above considerations that $(U_{n}%
)_{n\in\mathbb{N}}$ is a Cauchy sequence, so it converges to an element $U$,
which is a fixed point of $R\circ S$. The inequality (\ref{U_ineqcontr})
implies the uniqueness of this fixed point.
\end{proof}

The continuity of the application $S$ implies that the sequence $(u_{n}%
)_{n\in\mathbb{N}}$ converges to $u=S(U)$. It is easy to see that $(u,U)$ is
the unique solution of the problem (\ref{u1})-(\ref{U2}) on the interval
$[0,T_{0}]$.

\begin{theorem}
Let $0<l<1$, $\Omega\subset\mathbb{R}^{N}$ be a domain with $C^{l+2}$ boundary
$\partial\Omega$. Given an initial value $\left(  u_{0},w_{0}\right)
\in\left(  C^{l+2}(\overline{\Omega})\right)  ^{2}$, $u_{0}\geqslant0$,
$w_{0}>0$ and if the compatibility condition $\frac{\partial u_{0}}%
{\partial\eta}=u_{0}\frac{\partial w_{0}}{\partial\eta}$ is satisfied for
every $x\in\partial\Omega$, then the problem (\ref{eq1})-(\ref{ic2}) has a
unique nonnegative solution $\left(  u,w\right)  $ defined on an interval
$[0,T)\subset\mathbb{R}$ and $\left(  u,w\right)  \in\left(  C^{l+2,l/2+1}%
(\overline{\Omega}_{t})\right)  ^{2}$, for all $t\in\lbrack0,T)$.
\end{theorem}

\begin{proof}
Theorem \ref{contraction} implies the existence and the uniqueness of the
solution of the problem (\ref{eq1})-(\ref{ic2}) on $\overline{\Omega}%
\times\left[  0,\tau_{1}\right]  $ with $\tau_{1}$ sufficiently small. By
iterating the argument above, we can extend this solution on an interval
$\left[  \tau_{1},\tau_{2}\right]  $ and so on. At each step the conditions
$(i)$ and $(ii)$ in Theorem \ref{contraction} must be fulfilled and this
imposes restrictions on the length of the interval of existence. We emphasize
that this length depends continuously on the initial data, fact that will be
used in the next section for proving the global existence in time of the
solution. We obtain in such a way a solution defined in an interval
$[0,T)\subset\mathbb{R}$, $0<T\leqslant\infty$. The nonnegativity of the
solution results from the maximum principle.

In order to prove the uniqueness of the solution, it is enough to notice that
each classical solution of the problem (\ref{eq1})-(\ref{ic2}) can be
regarded, locally, as a fixed point of a map analogue to $R\circ S$. The
uniqueness of such a fixed point implies the uniqueness of the solution.
\end{proof}

\section{Global existence in time}

\label{global_existence} \ 

\setcounter{equation}{0} \setcounter{figure}{0}

In this Section we prove that the smooth solution of the problem
(\ref{eq1})-(\ref{ic2}) considered in a bounded domain $\Omega\subset
\mathbb{R}^{2}$ is globally defined in time. In order to do this, first we
derive some a priori estimates which then enable us to prove uniform
upper-bound for $\left\vert u\right\vert _{\Omega_{T}}^{(l+2)}$. Hereafter,
$T$ denotes the maximal existence time of the classical nonnegative solution
$(u,w)$ to (\ref{eq1})-(\ref{ic2}) obtained in Section \ref{local}%
\ corresponding to initial value $\left(  u_{0},w_{0}\right)  \in\left(
C^{l+2}(\overline{\Omega})\right)  ^{2}$.

The main result of this Section is:

\begin{theorem}
\label{globalex}Let $0<l<1$, and $\Omega\subset\mathbb{R}^{2}$ be a domain
with $C^{l+2}$ boundary $\partial\Omega$. Given an initial pair of functions
$\left(  u_{0},w_{0}\right)  \in\left(  C^{l+2}(\overline{\Omega})\right)
^{2}$, there exists a global in time nonnegative solution $(u,w)\in\left(
C^{l+2,l/2+1}(\overline{\Omega\times\lbrack0,\infty\lbrack})\right)  ^{2}$ to
the problem (\ref{eq1})-(\ref{ic2}).
\end{theorem}

We start by calculating a priori bounds that will be used for proving that the
solution $\left(  u,w\right)  $ to the system (\ref{eq1})-(\ref{ic2}) belongs
to a suitable H\"{o}lder space.

The regularity is then successively ameliorated until obtaining a uniform
bound of $\left\vert u(\cdot,t)\right\vert _{\Omega}^{(l+2)}$ by respect to
$t$. As the length of the existence interval obtained in Theorem
\ref{local existence} depends uniformly on $\left\vert u_{0}\right\vert
_{\Omega}^{(l+2)}$, this bound will imply that the maximal interval of
definition of the solution is $[0,\infty)$.

In what follows, sometimes the function arguments are omitted and for
simplicity we denote with $f_{t}$ the $t$-derivative of the function $f$.
Also, the variable $t$ belongs to the maximal time interval of existence of
the classical solution $(u,v)$ of the problem (\ref{eq1})-(\ref{ic2}).

\subsection{A Lyapunov function for the system}

\label{lyap}

The results obtained in this Subsection do not depend on the dimension of the
space, they are valid in a bounded domain $\Omega\subset\mathbb{R}^{N},$
$N\geqslant1$.

\begin{proposition}
\label{mass}Suppose that $\left\Vert u_{0}\right\Vert _{L^{1}(\Omega)}<\infty
$. Then the total mass of the solution $u$ is bounded%
\begin{equation}
\int\limits_{\Omega}u(x,t)dx\leqslant\left\vert \Omega\right\vert \max\left\{
1,M_{0}\right\}  \label{mass_estimat}%
\end{equation}
for all $t>0$, where $M_{0}=\left\vert \Omega\right\vert ^{-1}\left\Vert
u_{0}\right\Vert _{L^{1}\left(  \Omega\right)  }$ represents the initial mass
and $\left\vert \Omega\right\vert $ denotes the volume of $\Omega$.
\end{proposition}

\begin{proof}
Taking into account the boundary condition (\ref{bc}) and integrating the
equation (\ref{eq1}) over $\Omega$, we can easily deduce
\begin{equation}
\int\limits_{\Omega}u_{t}(x,t)dx=\delta\int\limits_{\Omega}u(x,t)dx-\delta
\int\limits_{\Omega}u^{2}(x,t)dx. \label{mass0}%
\end{equation}
Applying Jensen's inequality and Gronwall lemma we obtain the estimation
(\ref{mass_estimat}).
\end{proof}

\begin{remark}
1. Since the solution $u$ is nonnegative, a\ consequence of the property
(\ref{mass_estimat}) is that $u$ satisfies an a priori $L^{1}$ estimate
uniform in time
\[
\left\Vert u\right\Vert _{L^{\infty}(0,t;L^{1}\left(  \Omega\right)
)}=\left(  \left\Vert u^{1/2}\right\Vert _{L^{\infty}(0,t;L^{2}\left(
\Omega\right)  )}\right)  ^{2}\leqslant\left\vert \Omega\right\vert
\max\left\{  1,M_{0}\right\}
\]
for all $t>0$.

2. Let us observe that, from (\ref{eq2}), we have%
\begin{equation}
w=w_{0}e^{-\int\limits_{0}^{t}u\left[  w_{0}^{1-\beta}+\left(  \beta-1\right)
\int\limits_{0}^{s}u\right]  ^{-1}ds}. \label{wexpr}%
\end{equation}
For $w_{0}(x)>0$, $x\in\Omega$, we obtain $0<w(x,t)\leqslant w_{0}(x)$ for all
$t>0$, which implies%
\begin{equation}
\left\Vert w\right\Vert _{L^{\infty}(0,t;L^{\infty}(\Omega))}\leqslant
\left\Vert w_{0}\right\Vert _{L^{\infty}(\Omega)}. \label{winf}%
\end{equation}

\end{remark}

We introduce the following two functionals%
\begin{align}
F(u,w)  &  =\int\limits_{\Omega}u[\ln u-1]dx+\frac{1}{2}\int\limits_{\Omega
}w^{-\beta}\left\vert \nabla w\right\vert ^{2}dx,\label{Fly}\\
D(u,w)  &  =4\int\limits_{\Omega}\left\vert \nabla u^{1/2}\right\vert
^{2}dx+\frac{\beta}{2}\int\limits_{\Omega}uw^{-1}\left\vert \nabla
w\right\vert ^{2}dx+\delta\int\limits_{\Omega}u\left(  u-1\right)  \ln udx
\label{Dly}%
\end{align}
and we show that $F(u,w)$ is a Lyapunov functional to the system
(\ref{eq1})-(\ref{ic2}).

\begin{lemma}
\label{FD}If $(u,w)$ is a solutions to the system (\ref{eq1})-(\ref{ic2}),
then we have
\begin{equation}
\frac{d}{dt}F(u,w)=-D(u,w)\leqslant0. \label{Lyap1}%
\end{equation}

\end{lemma}

\begin{proof}
We formally differentiate the functional $F$ with respect to $t$:%
\[
\frac{d}{dt}F(u,w)=\int\limits_{\Omega}u_{t}[\ln u-1]dx+\int\limits_{\Omega
}u_{t}dx+\frac{1}{2}\frac{d}{dt}\int\limits_{\Omega}w^{-\beta}\left\vert
\nabla w\right\vert ^{2}dx.
\]
Multiplying the equation (\ref{eq1}) by $\ln u$ and formally integrating on
$\Omega$ (in fact we multiply by $\ln(u+\varepsilon)$, $\varepsilon>0$ and
after integration we make $\varepsilon\rightarrow0$), we get%
\[
\int\limits_{\Omega}u_{t}[\ln u-1]dx=\int\limits_{\Omega}\Delta u[\ln
u-1]dx-\int\limits_{\Omega}\nabla\cdot(u\nabla w)[\ln u-1]dx+\delta
\int\limits_{\Omega}u(1-u)[\ln u-1]dx
\]
and taking into account the equality (\ref{mass0}), we have
\begin{equation}
\int\limits_{\Omega}u_{t}[\ln u-1]dx=-4\int\limits_{\Omega}\left\vert \nabla
u^{1/2}\right\vert ^{2}dx+\int\limits_{\Omega}\nabla u\cdot\nabla
wdx+\delta\int\limits_{\Omega}u(1-u)\ln udx-\int\limits_{\Omega}u_{t}dx.
\label{esto}%
\end{equation}
Estimating the second term from the right-hand side in the last equality using
(\ref{eq2}):%
\[
\int\limits_{\Omega}\nabla u\cdot\nabla wdx=-\frac{1}{2}\frac{d}{dt}%
\int\limits_{\Omega}(w^{-\beta}\left\vert \nabla w\right\vert ^{2}%
)dx-\frac{\beta}{2}\int\limits_{\Omega}uw^{-1}\left\vert \nabla w\right\vert
^{2}dx
\]
and introducing it in (\ref{esto}), we obtain (\ref{Lyap1}).
\end{proof}

\ \ \ 

Throughout this paper we consider the following assumption on the initial data:

\ 

$(\mathcal{H})$ the functions $u_{0}(x)\geqslant0$ and $w_{0}(x)>0$ satisfy
$F(u_{0},w_{0})<+\infty$, for all $x\in\Omega$.

\ 

\begin{remark}
1. Let us observe that if the hypothesis $(\mathcal{H})$ is satisfied, then
$u_{0}\in L^{1}(\Omega)$ because%
\begin{equation}
\left\Vert u_{0}\right\Vert _{L^{1}(\Omega)}\leqslant\int\limits_{\Omega
}\left[  u_{0}\left(  \ln u_{0}-1\right)  +e\right]  dx\leqslant F(u_{0}%
,w_{0})+e\left\vert \Omega\right\vert . \label{u0}%
\end{equation}

2. In fact, because $u(\ln u-1)\geqslant-1$ for all $u>0$, the hypothesis
$(\mathcal{H})$ is equivalent with the boundedness of $\int\limits_{\Omega
}u_{0}\ln u_{0}dx$ and $w_{0}^{-\frac{\beta-2}{2}}\in H^{1}(\Omega)$ if
$\beta\neq2$ or $\ln w_{0}\in H^{1}(\Omega)$ if $\beta=2$.
\end{remark}

\begin{lemma}
\label{bounds}If the hypothesis $(\mathcal{H})$ is satisfied, then the
functional $F$ is bounded, i.e. there exists a positive constant $C_{4}$
independent on $t$, such that%
\begin{equation}
\left\vert F(u,w)\right\vert \leqslant C_{4} \label{Lyap2}%
\end{equation}
for all $t>0$. Moreover, the boundedness independently on $t$ of both terms of
the functional $F$ follows.
\end{lemma}

\begin{proof}
Integrating (\ref{Lyap1}) between $0$ and $t$, we obtain%
\begin{equation}
F(u,w)\leqslant F(u_{0},w_{0}). \label{fu}%
\end{equation}
Let us observe that for all $u>0$, $u\left(  \ln u-1\right)  >-1$ holds and we
have
\begin{equation}
F(u,w)=\int\limits_{\Omega}u[\ln u-1]dx+\frac{1}{2}\int\limits_{\Omega
}w^{-\beta}\left\vert \nabla w\right\vert ^{2}dx\geqslant-\left\vert
\Omega\right\vert . \label{F_delta}%
\end{equation}
From (\ref{fu}), (\ref{F_delta}) and taking into account also the hypothesis
$(\mathcal{H})$ we conclude the lemma with $C_{4}=\max\left\{  \left\vert
\Omega\right\vert ,\left\vert F(u_{0},w_{0})\right\vert \right\}  $.
\end{proof}

\begin{proposition}
If the hypothesis $(\mathcal{H})$ is satisfied, then there exists a positive
constant $C_{5}$ independent on $t$ such that%
\begin{equation}
\int\limits_{\Omega}u\ln udx<C_{5} \label{ulnu}%
\end{equation}
where $C_{5}=C_{5}(\int\limits_{\Omega}u_{0}\ln u_{0}dx,\left\Vert
w_{0}\right\Vert _{H^{1}\left(  \Omega\right)  }^{-\frac{\beta-2}{2}})$ if
$\beta\neq2$, or $C_{5}=C_{5}(\int\limits_{\Omega}u_{0}\ln u_{0},\left\Vert
\ln w_{0}\right\Vert _{H^{1}\left(  \Omega\right)  })$ if $\beta=2$.
\end{proposition}

\begin{proof}
Taking into account the estimates (\ref{mass_estimat}) and (\ref{fu}), we
have
\[
\int\limits_{\Omega}u\ln udx\leqslant\int\limits_{\Omega}u[\ln u-1]dx+\frac
{1}{2}\int\limits_{\Omega}w^{-\beta}\left\vert \nabla w\right\vert ^{2}%
dx+\int\limits_{\Omega}udx\leqslant F(u_{0},w_{0})+\left\vert \Omega
\right\vert \max\left\{  1,M_{0}\right\}  .
\]

\end{proof}

\begin{proposition}
\label{P1}If there exists a positive constant $C$, independent on $t$, such
that the positive function $u\ $satisfies
\[
\int\limits_{\Omega}u\ln udx<C
\]
then
\begin{equation}
\lim\limits_{k\rightarrow\infty}\left\Vert u_{k}\right\Vert _{L^{1}\left(
\Omega\right)  }=0 \label{uk}%
\end{equation}
uniformly by respect to $t>0$, where $u_{k}=(u-k)_{+}=\max\left\{
0,u-k\right\}  $, $k>0$.

\begin{proof}
For $k>1$ we obtain
\[
\left\Vert u_{k}\right\Vert _{L^{1}\left(  \Omega\right)  }\leqslant
\int\limits_{\substack{\Omega\\u(x)>k}}udx\leqslant\frac{1}{\ln k}%
\int\limits_{\substack{\Omega\\u(x)>k}}u\ln udx\leqslant\frac{1}{\ln k}\left(
\int\limits_{\Omega}u\ln udx-\int\limits_{\substack{\Omega\\u(x)<1}}u\ln
udx\right)  \leqslant C_{6}\frac{1}{\ln k}%
\]
where $C_{6}=\left(  C+e^{-1}\left\vert \Omega\right\vert \right)  $. The last
inequality implies (\ref{uk}).
\end{proof}
\end{proposition}

\subsection{$L^{p}$ a priori estimates, $1<p<\infty$}

\label{estimat2}

In order to obtain the desired $L^{p}$-bound on $u$, we make a change of
variables of the form%
\begin{equation}
v(x,t)=u(x,t)e^{-w(x,t)}. \label{uvw}%
\end{equation}
The system (\ref{eq1})-(\ref{ic2}) becomes%
\begin{align}
&  \frac{\partial v}{\partial t}=\Delta v+\nabla v\cdot\nabla w+e^{w}%
v^{2}w^{\beta}+\delta v(1-ve^{w}) & x  &  \in\Omega,\quad t\in\mathbb{R}%
_{+}\label{eq1v}\\
&  \frac{\partial w}{\partial t}=-e^{w}w^{\beta}v & x  &  \in\Omega,\quad
t\in\mathbb{R}_{+}\label{eq2v}\\
&  \frac{\partial v}{\partial\eta}=0 & x  &  \in\partial\Omega,\quad
t\in\mathbb{R}_{+}\label{bcv}\\
&  v(x,0)=u_{0}(x)e^{-w_{0}(x)}=v_{0}(x)\geqslant0 & x  &  \in\Omega
\label{ic1v}\\
&  w(x,0)=w_{0}(x)>0 & x  &  \in\Omega\label{ic2v}%
\end{align}
where $\delta\geqslant0$ and $\beta\geqslant1$.

\begin{remark}
\label{glob}We shall use this change of variables and the new
system\ (\ref{eq1v})-(\ref{ic2v}) in order to prove an uniform upper-bound for
$\left\vert v\right\vert _{\Omega_{T}}^{(l+2)}$ and subsequently to establish
an uniform upper-bound for $\left\vert u\right\vert _{\Omega_{T}}^{(l+2)}$.
\end{remark}

From now on, for simplicity of notation we shall write $v_{k}$ instead of
$\left(  v-k\right)  _{+}$, where $k>0$.

\begin{proposition}
\label{Lp}Let $\Omega\subset\mathbb{R}^{2}$ be a bounded domain with smooth
boundary. If the hypothesis $(\mathcal{H})$ is satisfied and $v_{0}\in
L^{p}(\Omega)$, $1<p<\infty$, $w_{0}\in L^{\infty}(\Omega)$, then there exists
a constant $C_{7}=C_{7}(p,\left\Vert v_{0}\right\Vert _{L^{p}(\Omega
)},\left\Vert w_{0}\right\Vert _{L^{\infty}(\Omega)})$ independent on time
such that the solution $v$ to the system (\ref{eq1v})-(\ref{ic2v}) satisfies%
\begin{equation}
\left\Vert v\right\Vert _{L^{\infty}(0,T;L^{p}(\Omega))}\leqslant C_{7}%
,\qquad\forall~1\leqslant p<+\infty. \label{lp}%
\end{equation}

\end{proposition}

\begin{proof}
Testing the equation (\ref{eq1v}) with $pv_{k}^{p-1}e^{w}$, $k>0$,
$p>1$,\ gives
\begin{equation}
\frac{d}{dt}%
{\displaystyle\int\limits_{\Omega}}
v_{k}^{p}e^{w}=-p(p-1)%
{\displaystyle\int\limits_{\Omega}}
v_{k}^{p-2}e^{w}\left\vert \nabla v_{k}\right\vert ^{2}+\delta p%
{\displaystyle\int\limits_{\Omega}}
v_{k}^{p-1}\left[  ve^{w}(1-ve^{w})\right]  +\left(  p-1\right)
{\displaystyle\int\limits_{\Omega}}
e^{2w}w^{\beta}v_{k}^{p+1}. \label{ess1}%
\end{equation}
Taking into account the identity%
\[
\left\vert \nabla\left(  v_{k}^{p/2}\right)  \right\vert ^{2}=\frac{p^{2}}%
{4}v_{k}^{p-2}\left\vert \nabla v_{k}\right\vert ^{2}%
\]
we obtain from (\ref{ess1})%
\begin{align}
\frac{d}{dt}%
{\displaystyle\int\limits_{\Omega}}
v_{k}^{p}e^{w}  &  =-\frac{4(p-1)}{p}%
{\displaystyle\int\limits_{\Omega}}
e^{w}\left\vert \nabla\left(  v_{k}^{p/2}\right)  \right\vert ^{2}+%
{\displaystyle\int\limits_{\Omega}}
pke^{w}\left[  ke^{w}w^{\beta}+\delta\left(  1-ke^{w}\right)  \right]
v_{k}^{p-1}+\nonumber\\
&  +%
{\displaystyle\int\limits_{\Omega}}
e^{w}\left[  (2p-1)ke^{w}w^{\beta}+\delta p\left(  1-2ke^{w}\right)  \right]
v_{k}^{p}+%
{\displaystyle\int\limits_{\Omega}}
e^{2w}\left[  (p-1)w^{\beta}-\delta p\right]  v_{k}^{p+1}. \label{ess2}%
\end{align}
Since $0<w(x,t)\leqslant w_{0}(x)$ and $e^{w(x,t)}\geqslant1$ for all
$x\in\Omega$, $t>0$, we obtain from (\ref{ess2})%
\begin{equation}
\frac{d}{dt}%
{\displaystyle\int\limits_{\Omega}}
v_{k}^{p}e^{w}\leqslant-\frac{4(p-1)}{p}\left\Vert \nabla\left(  v_{k}%
^{p/2}\right)  \right\Vert _{L^{2}(\Omega)}^{2}+C_{8}%
{\displaystyle\int\limits_{\Omega}}
v_{k}^{p-1}+C_{9}%
{\displaystyle\int\limits_{\Omega}}
v_{k}^{p}+C_{10}%
{\displaystyle\int\limits_{\Omega}}
v_{k}^{p+1} \label{ess3}%
\end{equation}
where we have made the following notations%
\begin{align}
C_{8}  &  =C_{8}(p,k,\delta,\left\Vert w_{0}\right\Vert _{L^{\infty}\left(
\Omega\right)  })=\nonumber\\
&  =pk\left[  k\left(  e^{2\left\Vert w_{0}\right\Vert _{L^{\infty}\left(
\Omega\right)  }}\left\Vert w_{0}\right\Vert _{L^{\infty}\left(
\Omega\right)  }^{\beta}-\delta\frac{p}{p-1}\right)  +\delta\left(
e^{\left\Vert w_{0}\right\Vert _{L^{\infty}\left(  \Omega\right)  }}+\frac
{k}{p-1}\right)  \right] \label{c8}\\
C_{9}  &  =C_{9}(p,k,\delta,\left\Vert w_{0}\right\Vert _{L^{\infty}\left(
\Omega\right)  })=\nonumber\\
&  =p\left[  \frac{2p-1}{p}k\left(  e^{2\left\Vert w_{0}\right\Vert
_{L^{\infty}\left(  \Omega\right)  }}\left\Vert w_{0}\right\Vert _{L^{\infty
}\left(  \Omega\right)  }^{\beta}-\delta\frac{p}{p-1}\right)  +\delta\left(
e^{\left\Vert w_{0}\right\Vert _{L^{\infty}\left(  \Omega\right)  }}+\frac
{k}{p-1}\right)  \right] \label{c9}\\
C_{10}  &  =C_{10}(p,\delta,\left\Vert w_{0}\right\Vert _{L^{\infty}\left(
\Omega\right)  })=\left(  p-1\right)  \left(  e^{2\left\Vert w_{0}\right\Vert
_{L^{\infty}\left(  \Omega\right)  }}\left\Vert w_{0}\right\Vert _{L^{\infty
}\left(  \Omega\right)  }^{\beta}-\delta\frac{p}{p-1}\right)  . \label{c10}%
\end{align}
Adding the term $\sigma%
{\displaystyle\int\limits_{\Omega}}
v_{k}^{p}$, where $\sigma>0$ is a constant, on both sides of the last
inequality, we obtain
\begin{equation}
\frac{d}{dt}%
{\displaystyle\int\limits_{\Omega}}
v_{k}^{p}e^{w}+\sigma%
{\displaystyle\int\limits_{\Omega}}
v_{k}^{p}\leqslant-\frac{4(p-1)}{p}\left\Vert \nabla\left(  v_{k}%
^{p/2}\right)  \right\Vert _{L^{2}(\Omega)}^{2}+C_{8}%
{\displaystyle\int\limits_{\Omega}}
v_{k}^{p-1}+\left(  C_{9}+\sigma\right)
{\displaystyle\int\limits_{\Omega}}
v_{k}^{p}+C_{10}%
{\displaystyle\int\limits_{\Omega}}
v_{k}^{p+1}. \label{ess45}%
\end{equation}
We estimate now the last two terms from (\ref{ess45}) using
Gagliardo-Nirenberg's inequality and taking into account the positivity of
$v$. We have%
\begin{align}
&
{\displaystyle\int\limits_{\Omega}}
v_{k}^{p}=\left\Vert v_{k}^{\frac{p}{2}}\right\Vert _{L^{2}(\Omega)}%
^{2}\leqslant C_{11}(\Omega)\left\Vert v_{k}^{\frac{p}{2}}\right\Vert
_{H^{1}(\Omega)}\left\Vert v_{k}^{\frac{p}{2}}\right\Vert _{L^{1}(\Omega
)},\label{ess5}\\
&
{\displaystyle\int\limits_{\Omega}}
v_{k}^{p+1}=\left\Vert v_{k}^{\frac{p}{2}}\right\Vert _{L^{\frac{2(p+1)}{p}%
}(\Omega)}^{\frac{2(p+1)}{p}}\leqslant C_{12}(\Omega)\left\Vert v_{k}%
^{\frac{p}{2}}\right\Vert _{H^{1}(\Omega)}^{2}\left\Vert v_{k}\right\Vert
_{L^{1}(\Omega)}. \label{ess6}%
\end{align}
We insert the estimations (\ref{ess5}), (\ref{ess6}) into (\ref{ess45}) and we
apply Cauchy's inequality. We obtain%
\begin{align}
\frac{d}{dt}%
{\displaystyle\int\limits_{\Omega}}
v_{k}^{p}e^{w}+\sigma%
{\displaystyle\int\limits_{\Omega}}
v_{k}^{p}  &  \leqslant-\frac{4(p-1)}{p}\left\Vert \nabla\left(  v_{k}%
^{\frac{p}{2}}\right)  \right\Vert _{L^{2}(\Omega)}^{2}+C_{8}\left\Vert
v_{k}^{p-1}\right\Vert _{L^{1}(\Omega)}+\nonumber\\
&  +\left[  C_{9}+\sigma\right]  \left\Vert v_{k}^{\frac{p}{2}}\right\Vert
_{L^{2}(\Omega)}^{2}+C_{10}C_{12}\left\Vert v_{k}^{\frac{p}{2}}\right\Vert
_{H^{1}(\Omega)}^{2}\left\Vert v_{k}\right\Vert _{L^{1}(\Omega)}
\leqslant\nonumber\\
&  \leqslant\left[  -\frac{4(p-1)}{p}+C_{10}C_{12}\left\Vert v_{k}\right\Vert
_{L^{1}(\Omega)}+\varepsilon\right]  \left\Vert \nabla\left(  v_{k}^{\frac
{p}{2}}\right)  \right\Vert _{L^{2}(\Omega)}^{2}+C_{8}\left\Vert v_{k}%
^{p-1}\right\Vert _{L^{1}(\Omega)}+\nonumber\\
&  +\varepsilon\left\Vert v_{k}^{\frac{p}{2}}\right\Vert _{L^{2}(\Omega)}%
^{2}+\frac{1}{4\varepsilon}\left\{  C_{11}\left[  C_{9}+\sigma+C_{10}%
C_{12}\left\Vert v_{k}\right\Vert _{L^{1}(\Omega)}\right]  \right\}
^{2}\left\Vert v_{k}^{\frac{p}{2}}\right\Vert _{L^{1}(\Omega)}^{2}.
\label{ess65}%
\end{align}
In order to estimate the second term from the right-hand side of
(\ref{ess65}), we apply Young's inequality and we obtain for $\epsilon>0$%
\begin{equation}
v_{k}^{p-1}\leqslant\frac{1}{p}\epsilon^{-p}+\frac{p-1}{p}\epsilon^{\frac
{p}{p-1}}v_{k}^{p}. \label{ess7}%
\end{equation}
Now, choosing $\varepsilon$ small enough such that $\varepsilon<\min\left\{
\sigma/2,2(p-1)/p\right\}  $ and inserting (\ref{ess7}) in (\ref{ess65}), we
get%
\begin{align*}
\frac{d}{dt}%
{\displaystyle\int\limits_{\Omega}}
v_{k}^{p}e^{w}+\frac{\sigma}{2}%
{\displaystyle\int\limits_{\Omega}}
v_{k}^{p}  &  \leqslant\left[  -\frac{4(p-1)}{p}+C_{10}C_{12}\left\Vert
v_{k}\right\Vert _{L^{1}(\Omega)}+\varepsilon\right]  \left\Vert \nabla\left(
v_{k}^{\frac{p}{2}}\right)  \right\Vert _{L^{2}(\Omega)}^{2}+\frac{1}%
{p}\epsilon^{-p}C_{8}\left\vert \Omega\right\vert +\\
&  +\frac{1}{4\varepsilon}\left\{  C_{11}\left[  C_{9}+\sigma+\frac{p-1}%
{p}\epsilon^{\frac{p}{p-1}}C_{8}+C_{10}C_{12}\left\Vert v_{k}\right\Vert
_{L^{1}(\Omega)}\right]  \right\}  ^{2}\left\Vert v_{k}^{\frac{p}{2}%
}\right\Vert _{L^{1}(\Omega)}^{2}.
\end{align*}
Taking into account Proposition \ref{P1}, we can choose $k$ sufficiently large
such that the coefficient of $\left\Vert \nabla\left(  v_{k}^{\frac{p}{2}%
}\right)  \right\Vert _{L^{2}(\Omega)}^{2}$ is negative. In this way, using
(\ref{winf}), the last inequality becomes%
\begin{equation}
\frac{d}{dt}%
{\displaystyle\int\limits_{\Omega}}
v_{k}^{p}e^{w}+\frac{\sigma}{2e^{\left\Vert w_{0}\right\Vert _{L^{\infty
}\left(  \Omega\right)  }}}%
{\displaystyle\int\limits_{\Omega}}
v_{k}^{p}e^{w}\leqslant C_{13}\left\Vert v_{k}^{\frac{p}{2}}\right\Vert
_{L^{1}(\Omega)}^{2}+\frac{1}{p}\epsilon^{-p}C_{8}\left\vert \Omega\right\vert
\label{ess8}%
\end{equation}
where%
\[
C_{13}=C_{13}(p,k,\delta,\left\Vert w_{0}\right\Vert _{L^{\infty}\left(
\Omega\right)  },\sigma,\varepsilon,\epsilon,\Omega)=\frac{1}{4\varepsilon
}\left\{  C_{11}\left[  C_{9}+\sigma+\frac{p-1}{p}\epsilon^{\frac{p}{p-1}%
}C_{8}+C_{10}C_{12}\left\Vert v_{k}\right\Vert _{L^{1}(\Omega)}\right]
\right\}  ^{2}.
\]
Applying Gronwall's inequality, we obtain \ from (\ref{ess8})
\begin{equation}%
{\displaystyle\int\limits_{\Omega}}
v_{k}^{p}\leqslant%
{\displaystyle\int\limits_{\Omega}}
v_{k}^{p}e^{w}\leqslant\max\left\{
{\displaystyle\int\limits_{\Omega}}
\left(  v_{0}-k\right)  _{+}^{p}e^{w_{0}},\frac{2e^{\left\Vert w_{0}%
\right\Vert _{L^{\infty}\left(  \Omega\right)  }}}{\sigma}\left[
C_{13}\left\Vert v_{k}^{\frac{p}{2}}\right\Vert _{L^{1}(\Omega)}^{2}+\frac
{1}{p}\epsilon^{-p}C_{8}\left\vert \Omega\right\vert \right]  \right\}  .
\label{vk}%
\end{equation}
We will show by induction that
\[
\left\Vert v_{k}(t)\right\Vert _{L^{p}(\Omega)}\leqslant C
\]
for all $p=2^{j}$, with $j\in\mathbb{N}$, where $C$ is a constant independent
of $t$.

Let us remark that, taking into account Proposition \ref{mass}, we have%
\begin{equation}
\left\Vert v_{k}(t)\right\Vert _{L^{1}(\Omega)}\leqslant\left\vert
\Omega\right\vert \max\left\{  1,M_{0}\right\}  . \label{vkl1}%
\end{equation}
Let $p=2^{j}$, and suppose that $\left\Vert v_{k}(t)\right\Vert _{L^{2^{j-1}%
}(\Omega)}=\left\Vert v_{k}(t)\right\Vert _{L^{p/2}(\Omega)}$ is uniformly
bounded, the bound being independent of $t>0$. We obtain from (\ref{vk}) that
$\left\Vert v_{k}(t)\right\Vert _{L^{2^{j}}(\Omega)}$ is bounded,
$j\in\mathbb{N}\backslash\left\{  0\right\}  $. We conclude, taking into
account the embeddings of $L^{p}\left(  \Omega\right)  $ spaces, that
\[
\left\Vert v_{k}\right\Vert _{L^{\infty}(0,t;L^{p}(\Omega))}\leqslant
C_{14}\text{, for every }1\leqslant p<\infty
\]
where $C_{14}=C_{14}(p,\left\Vert v_{0}\right\Vert _{L^{1}(\Omega)},\left\Vert
w_{0}\right\Vert _{L^{\infty}\left(  \Omega\right)  })$\ is a positive
constant, independent of $t>0$.

Finally, we obtain%
\[
\left\Vert v\right\Vert _{L^{p}(\Omega)}\leqslant2\left(  \left\Vert
v_{k}(t)\right\Vert _{L^{p}(\Omega)}^{p}+k^{p}\left\vert \Omega\right\vert
\right)  ^{1/p}%
\]
and we conclude the Theorem.
\end{proof}

\begin{remark}
The above estimations are strongly dependent on the dimension of the space and
they are done in the case when $C_{8},$ $C_{9},$ $C_{10}$ are positive. If one
or several of these constants are negative (for example, when $\delta
>\frac{p-1}{p}e^{2\left\Vert w_{0}\right\Vert _{L^{\infty}\left(
\Omega\right)  }}\left\Vert w_{0}\right\Vert _{L^{\infty}\left(
\Omega\right)  }^{\beta}$, $p>1$), the result remains true the upper bound
being slightly different.
\end{remark}

\subsection{$L^{\infty}$ a priori estimates}

\label{estimat3}

\begin{proposition}
\label{estimat_infinity}Let $\Omega\subset\mathbb{R}^{2}$ be a bounded domain
with smooth boundary. If the hypothesis $(\mathcal{H})$ is satisfied,
$v_{0}\in L^{\infty}(\Omega)$ and $w_{0}\in L^{\infty}(\Omega)$, then the
solution $v$ of the system (\ref{eq1v})-(\ref{ic2v}) satisfies%
\[
\left\Vert v\right\Vert _{L^{\infty}(0,T;L^{\infty}(\Omega))}\leqslant C
\]
where $C$ is a positive constant independent on time which will be determined later.
\end{proposition}

\begin{proof}
We introduce the following sets%
\[
\Omega_{k}(t)=\left\{  x\in\Omega;\quad v(x,t)>k\right\}
\]
where $k$ is a positive constant. Let us observe that, taking into account
(\ref{vkl1}) and choosing $p=2$, the relation (\ref{ess65}) becomes%
\begin{align}
\frac{d}{dt}%
{\displaystyle\int\limits_{\Omega}}
v_{k}^{2}e^{w}+\sigma%
{\displaystyle\int\limits_{\Omega}}
v_{k}^{2}  &  \leqslant\left[  -2+C_{10}C_{12}\left\Vert v_{k}\right\Vert
_{L^{1}(\Omega)}+\varepsilon\right]  \left\Vert \nabla v_{k}\right\Vert
_{L^{2}(\Omega)}^{2}+\varepsilon\left\Vert v_{k}\right\Vert _{L^{2}(\Omega
)}^{2}+\nonumber\\
&  +\left\{  C_{8}+\frac{1}{4\varepsilon}\left\Vert v_{k}\right\Vert
_{L^{1}(\Omega)}\left[  C_{11}\left(  C_{9}+\sigma+C_{10}C_{12}\left\Vert
v_{k}\right\Vert _{L^{1}(\Omega)}\right)  \right]  ^{2}\right\}  \left\Vert
v_{k}\right\Vert _{L^{1}(\Omega)}. \label{ess9}%
\end{align}
We estimate the last term of the right-hand side of the last inequality using
H\"{o}lder's inequality and\ a Sobolev embedding%
\[
\left\Vert v_{k}\right\Vert _{L^{1}(\Omega)}\leqslant\left\Vert v_{k}%
\right\Vert _{L^{4}(\Omega)}\left\vert \Omega_{k}\right\vert ^{3/4}\leqslant
C_{15}\left\Vert v_{k}\right\Vert _{H^{1}(\Omega)}\left\vert \Omega
_{k}\right\vert ^{3/4}%
\]
where $C_{15}$ is a constant independent of $t$. Using this inequality and
Cauchy's inequality, we obtain from (\ref{ess9})%
\begin{align}
\frac{d}{dt}%
{\displaystyle\int\limits_{\Omega}}
v_{k}^{2}e^{w}+\sigma%
{\displaystyle\int\limits_{\Omega}}
v_{k}^{2}  &  \leqslant\left[  -2+C_{10}C_{12}\left\Vert v_{k}\right\Vert
_{L^{1}(\Omega)}+\varepsilon+\varepsilon^{\prime}\right]  \left\Vert \nabla
v_{k}\right\Vert _{L^{2}(\Omega)}^{2}+\left(  \varepsilon+\varepsilon^{\prime
}\right)  \left\Vert v_{k}\right\Vert _{L^{2}(\Omega)}^{2}+\nonumber\\
&  +\frac{C_{15}^{2}}{4\varepsilon^{\prime}}\left\{  C_{8}+\frac
{1}{4\varepsilon}\left\Vert v_{k}\right\Vert _{L^{1}(\Omega)}\left[
C_{11}\left(  C_{9}+\sigma+C_{10}C_{12}\left\Vert v_{k}\right\Vert
_{L^{1}(\Omega)}\right)  \right]  ^{2}\right\}  ^{2}\left\vert \Omega
_{k}\right\vert ^{3/2}. \label{ess10}%
\end{align}
We choose $\varepsilon$ and $\varepsilon^{\prime}$ small enough such that
$\varepsilon+\varepsilon^{\prime}<\min\left\{  1,\sigma/2\right\}  $. Taking
into account Proposition \ref{P1}, it follows that there exists $k_{1}>0$
sufficiently large such that, for every $k>k_{1}$, the coefficient of
$\left\Vert \nabla v_{k}\right\Vert _{L^{2}(\Omega)}^{2}$ is negative. Taking
into account (\ref{winf}) and (\ref{vkl1}) we obtain from (\ref{ess10})%
\begin{equation}
\frac{d}{dt}%
{\displaystyle\int\limits_{\Omega}}
v_{k}^{2}e^{w}+\frac{\sigma}{2e^{\left\Vert w_{0}\right\Vert _{L^{\infty
}\left(  \Omega\right)  }}}%
{\displaystyle\int\limits_{\Omega}}
v_{k}^{2}e^{w}\leqslant C_{16}\left\vert \Omega_{k}\right\vert ^{3/2}
\label{c16n}%
\end{equation}
for all $k>k_{1}$, where%
\begin{align*}
C_{16}  &  =C_{16}(k,\delta,\left\Vert w_{0}\right\Vert _{L^{\infty}\left(
\Omega\right)  },M_{0})=\\
&  =\frac{C_{15}^{2}}{4\varepsilon^{\prime}}\left\{  C_{8}+\frac
{1}{4\varepsilon}\left\vert \Omega\right\vert \max\left\{  1,M_{0}\right\}
\left[  C_{11}\left(  C_{9}+\sigma+C_{10}C_{12}\left\vert \Omega\right\vert
\max\left\{  1,M_{0}\right\}  \right)  \right]  ^{2}\right\}  ^{2}.
\end{align*}
One can notice, using (\ref{c8}), (\ref{c9}) and (\ref{c10}), that $C_{16}%
$\ is a polynomial of degree $4$ in $k$. In the first place we shall focus on
obtaining an inequality similar to (\ref{c16n}) where the constant appearing
in the right-hand side is independent on $k$. Let $\alpha$ be the dominant
coefficient of $C_{16}$ as a polynomial in $k$. It is a constant depending
only on the initial data of the system. On the other hand, we have (see
\cite{rey})
\[
\int\limits_{\Omega}v^{q+1}=(q+1)\int\limits_{0}^{\infty}s^{q}\left\vert
\Omega_{s}\right\vert ds,\quad q\geqslant1.
\]
We obtain, using these facts, a bound for the right-hand side of the
inequality (\ref{c16n}). Namely, taking into account Proposition \ref{Lp}, we
get first%
\[
(k-1)^{q}\left\vert \Omega_{k}\right\vert <\int\limits_{k-1}^{k}%
s^{q}\left\vert \Omega_{s}\right\vert ds<\int\limits_{0}^{\infty}%
s^{q}\left\vert \Omega_{s}\right\vert ds=\frac{1}{q+1}\left\Vert v\right\Vert
_{L^{q+1}(\Omega)}^{q+1}<C_{17}%
\]
where $C_{17}$ is a constant independent on $k$ and on $t$. From the last
inequality, taking $q=16$, we obtain%
\[
(k-1)^{4}\left\vert \Omega_{k}\right\vert ^{1/4}<C_{17}^{1/4}.
\]
It follows that there exists $k_{2}>0$ such that for every $k>k_{2}$,
\[
C_{16}\left\vert \Omega_{k}\right\vert ^{1/4}<(\alpha+1)C_{17}^{1/4}=C_{18}%
\]
which implies, from (\ref{c16n})%
\begin{equation}
\frac{d}{dt}%
{\displaystyle\int\limits_{\Omega}}
v_{k}^{2}e^{w}+\frac{\sigma}{2e^{\left\Vert w_{0}\right\Vert _{L^{\infty
}\left(  \Omega\right)  }}}%
{\displaystyle\int\limits_{\Omega}}
v_{k}^{2}e^{w}\leqslant C_{18}\left\vert \Omega_{k}\right\vert ^{5/4}.
\label{c18}%
\end{equation}
In this way we obtained an inequality similar to (\ref{c16n}) where the
constant $C_{18}$ does not depend on $k$.

Since $v_{0}\in L^{\infty}(\Omega)$, there exists $k_{3}>0$ such that
$\left\Vert v_{k}(0)\right\Vert _{_{L^{\infty}\left(  \Omega\right)  }}=0$ for
all $k>k_{3}$. For $k>\max\{k_{1},k_{2},k_{3}\}$, we deduce from (\ref{c18})
using Gronwall's inequality
\begin{equation}
\left\Vert v_{k}(t)\right\Vert _{L^{2}(\Omega)}^{2}\leqslant\left\Vert
e^{w/2}v_{k}(t)\right\Vert _{L^{2}(\Omega)}^{2}\leqslant\frac{2C_{18}%
e^{\left\Vert w_{0}\right\Vert _{L^{\infty}\left(  \Omega\right)  }}}{\sigma
}\left(  1-e^{-\frac{\sigma}{2e^{\left\Vert w_{0}\right\Vert _{L^{\infty
}\left(  \Omega\right)  }}}t}\right)  \left(  \sup_{t\geqslant0}\left\vert
\Omega_{k}(t)\right\vert \right)  ^{5/4}. \label{vsup1}%
\end{equation}
On the other hand, taking into account that $\Omega_{l}\subset\Omega_{k}$ for
$l>k>0$,
\begin{equation}
\left\Vert v_{k}(t)\right\Vert _{L^{2}(\Omega)}^{2}\geqslant\int
\limits_{\Omega_{l}(t)}v_{k}^{2}\geqslant\left(  l-k\right)  ^{2}\left\vert
\Omega_{l}(t)\right\vert . \label{vsup2}%
\end{equation}
Taking the supremum on $t\geqslant0$ in the last relation, (\ref{vsup1})
implies%
\[
\left(  l-k\right)  ^{2}\sup_{t\geqslant0}\left\vert \Omega_{l}(t)\right\vert
\leqslant\frac{2C_{18}e^{\left\Vert w_{0}\right\Vert _{L^{\infty}\left(
\Omega\right)  }}}{\sigma}\left(  \sup_{t\geqslant0}\left\vert \Omega
_{k}(t)\right\vert \right)  ^{5/4}%
\]
for $l>k>\max\{k_{1},k_{2},k_{3}\}$. Obviously the function $k\mapsto
\sup_{t\geqslant0}\left\vert \Omega_{k}(t)\right\vert $ is decreasing, so we
can apply Lemma 4.1\ from \cite{fang}. It follows that there exists
\[
k_{0}=\max\{k_{1},k_{2},k_{3}\}+\left(  \frac{2^{11}C_{18}e^{\left\Vert
w_{0}\right\Vert _{L^{\infty}\left(  \Omega\right)  }}}{\sigma}\right)
^{1/2}\left\vert \Omega\right\vert ^{1/8}%
\]
such that%
\[
\sup_{t\geqslant0}\left\vert \Omega_{k}(t)\right\vert =0
\]
for all $k\geqslant k_{0}$. This concludes the proof.
\end{proof}

\begin{remark}
The $L^{\infty}$ bound can also be proved using the iterative technique of
Alikakos \cite{ali}. We have chosen the method presented here (inspired by an
idea of Gajewski and Zacharias \cite{gaj}) mainly for aesthetic reasons.

It is obvious that the conclusion of Proposition \ref{estimat_infinity}
remains valid also in the case of the classical solution $u$ of the system
(\ref{eq1})-(\ref{ic2}).
\end{remark}

\subsection{A priori estimates for $\nabla v$, $\nabla w$ and $\Delta v$}

\label{estimat4}

Taking the initial data in $(W^{2,q})^{2}$, $q>2$, in \cite{guar} the authors
derive $L^{\infty}(0,t;L^{p}(\Omega))$, $1<p\leqslant q$ bounds for $\nabla
v$, $\nabla w$ and $\Delta v$. Based on these estimates, under $L^{\infty}$
bounds assumptions, they show the global existence of weak solutions.
Moreover, under the same hypotheses on the initial data, it is proved that the
solution has some regularity properties.

By a different strategy we establish hereafter a priori bounds for $\left\Vert
\nabla v\right\Vert _{L^{1}(0,t;L^{p}(\Omega))}$, $\left\Vert \Delta
v\right\Vert _{L^{1}(0,t;L^{2}(\Omega))}$ and $\sup\limits_{t\in\left[
0,T\right]  }\left\Vert \nabla w\right\Vert _{L^{p}(\Omega)}$, $p\geqslant2$.
We mention that both lines of computation could be applied, as an intermediary
step, in order to obtain classical solutions. However, using the a priori
bounds which are given in what follows, one may prove the existence of the
weak solutions (in the sense of \cite{guar}) of the problem (\ref{eq1}%
)-(\ref{ic2}) starting with the initial data in, for example, $\left(
H^{1}(\Omega)\cap L^{\infty}(\Omega)\right)  \times W^{1,4}(\Omega)$.

\begin{lemma}
\label{p7}Let $\Omega\subset\mathbb{R}^{2}$ be a bounded domain. If the
hypothesis $(\mathcal{H})$ is satisfied, $v_{0}\in H^{1}(\Omega)\cap
L^{\infty}(\Omega)$ and $w_{0}\in L^{\infty}(\Omega)$, we have%
\begin{align}
\left\Vert v_{t}\right\Vert _{L^{2}(\Omega_{T})}  &  \leqslant C_{20}%
\label{vs}\\
\left\Vert \nabla v\right\Vert _{L^{2}(\Omega)}  &  \leqslant C_{20}\text{ }
\label{grv}%
\end{align}
for all $t>0$, where $C_{20}$\ is a constant independent on $t$.
\end{lemma}

\begin{proof}
Taking $e^{w}v_{t}$ as a test function in the equation (\ref{eq1v}) and
integrating in space, we obtain%
\begin{equation}
\int\limits_{\Omega}e^{w}v_{t}^{2}+\frac{1}{2}\frac{d}{dt}\int\limits_{\Omega
}\left(  e^{w}\left\vert \nabla v\right\vert ^{2}\right)  =-\frac{1}{2}%
\int\limits_{\Omega}e^{2w}vw^{\beta}\left\vert \nabla v\right\vert ^{2}%
+\int\limits_{\Omega}e^{w}\left[  e^{w}v^{2}w^{\beta}+\delta v\left(
1-ve^{w}\right)  \right]  v_{t}. \label{a11}%
\end{equation}
In order to estimate the last term from the right-hand side of (\ref{a11}) we
take into account the following inequalities
\begin{align}
&  \int\limits_{\Omega}e^{2w}w^{\beta}v^{2}v_{t}\leqslant\frac{1}{2}%
\int\limits_{\Omega}e^{w}v_{t}^{2}-\frac{1}{2}e^{2\left\Vert w_{0}\right\Vert
_{L^{\infty}\left(  \Omega\right)  }}\left\Vert w_{0}\right\Vert _{L^{\infty
}\left(  \Omega\right)  }\left\Vert v\right\Vert _{L^{\infty}(\Omega)}^{3}%
\int\limits_{\Omega}\frac{\partial w}{\partial t}\label{1_P}\\
&  \delta\int\limits_{\Omega}e^{w}vv_{t}\leqslant\frac{\delta}{2}\frac{d}%
{dt}\int\limits_{\Omega}e^{w}v^{2}-\frac{\delta}{2}e^{\left\Vert
w_{0}\right\Vert _{L^{\infty}(\Omega)}}\left\Vert v\right\Vert _{L^{\infty
}(\Omega)}^{2}\int\limits_{\Omega}\frac{\partial w}{\partial t}\label{2_P}\\
&  -\delta\int\limits_{\Omega}e^{2w}v^{2}v_{t}\leqslant-\frac{\delta}{3}%
\frac{d}{dt}\int\limits_{\Omega}e^{2w}v^{3}. \label{3_P}%
\end{align}
Substituting (\ref{1_P}), (\ref{2_P}) and (\ref{3_P}) into (\ref{a11}), after
that integrating in time and taking into account (\ref{winf}) we obtain%
\begin{equation}
\int\limits_{0}^{t}\int\limits_{\Omega}v_{s}^{2}+\int\limits_{\Omega}%
e^{w}\left\vert \nabla v\right\vert ^{2}\leqslant\int\limits_{\Omega}e^{w_{0}%
}\left\vert \nabla v_{0}\right\vert ^{2}+C_{19} \label{vsest}%
\end{equation}
where%
\begin{align*}
C_{19}  &  =e^{\left\Vert w_{0}\right\Vert _{L^{\infty}(\Omega)}}\left\Vert
v\right\Vert _{L^{\infty}(0,t;L^{\infty}(\Omega))}^{2}\left[  e^{\left\Vert
w_{0}\right\Vert _{L^{\infty}(\Omega)}}\left\Vert w_{0}\right\Vert
_{L^{\infty}(\Omega)}^{\beta}\left\Vert v\right\Vert _{L^{\infty
}(0,t;L^{\infty}(\Omega))}+\delta\right]  \left\Vert w_{0}\right\Vert
_{L^{1}(\Omega)}+\\
&  +\delta e^{\left\Vert w_{0}\right\Vert _{L^{\infty}(\Omega)}}\left\Vert
v\right\Vert _{L^{2}(\Omega)}^{2}+\frac{2\delta}{3}\int\limits_{\Omega
}e^{2w_{0}}v_{0}^{3}.
\end{align*}
The last inequality implies (\ref{vs}) and (\ref{grv}) where $C_{20}=\left(
e^{\left\Vert w_{0}\right\Vert _{L^{\infty}(\Omega)}}\left\Vert \nabla
v_{0}\right\Vert _{L^{2}(\Omega)}^{2}+C_{19}\right)  ^{1/2}$.
\end{proof}

\begin{lemma}
\label{p3.}Let $\Omega\subset\mathbb{R}^{2}$ be a bounded domain. If the
hypothesis $(\mathcal{H})$ is satisfied, $v_{0}\in H^{1}(\Omega)\cap
L^{\infty}(\Omega)$ and $w_{0}\in W^{1,4}(\Omega)$, we have
\begin{equation}
\left\Vert \Delta v\right\Vert _{L^{1}(0,t;L^{2}(\Omega))}\leqslant e\left(
n+1\right)  !k(T_{0}) \label{lap}%
\end{equation}
for all $t\in\left[  0,\min\left\{  \left(  n+1\right)  T_{0},T\right\}
\right]  $, $n\in\mathbb{N}\backslash\left\{  0\right\}  $, where $T_{0}$ is a
constant independent on $t$ and $k$ is a function with liniar growing which
will be given later.
\end{lemma}

\begin{proof}
From (\ref{eq1v}) we obtain for every $0\leqslant t<T$%
\begin{equation}
\int\limits_{0}^{t}\left\Vert \Delta v\right\Vert _{L^{2}(\Omega)}%
\leqslant\int\limits_{0}^{t}\left\Vert v_{s}\right\Vert _{L^{2}(\Omega)}%
+\int\limits_{0}^{t}\left\Vert \nabla w\cdot\nabla v\right\Vert _{L^{2}%
(\Omega)}+\int\limits_{0}^{t}\left\Vert h(v,w)\right\Vert _{L^{2}(\Omega)}
\label{de}%
\end{equation}
where%
\[
h(v,w)=e^{w}v^{2}w^{\beta}+\delta v(1-ve^{w}).
\]
We estimate the first term from (\ref{de}) using (\ref{vs}) and the H\"{o}lder
inequality%
\begin{equation}
\int\limits_{0}^{t}\left\Vert v_{s}\right\Vert _{L^{2}(\Omega)}\leqslant
t^{1/2}\left(  \int\limits_{0}^{t}\int\limits_{\Omega}\left\vert
v_{s}\right\vert ^{2}\right)  ^{1/2}\leqslant\sqrt{C_{20}}\,t^{1/2}.
\label{t1}%
\end{equation}
In order to obtain an estimate for $\left\Vert \nabla w\right\Vert
_{L^{4}(\Omega)}$, we deduce from the equation (\ref{eq2v})
\[
\nabla w_{t}=-e^{w}w^{\beta}v\nabla w-\beta e^{w}w^{\beta-1}v\nabla
w-e^{w}w^{\beta}\nabla v.
\]
Multiplying the last relation by $\nabla w\left\vert \nabla w\right\vert ^{2}$
we obtain by integration%
\begin{equation}
\left\Vert \nabla w\right\Vert _{L^{4}(\Omega)}\leqslant C_{22}+C_{21}%
\int\limits_{0}^{t}\left\Vert \nabla v\right\Vert _{L^{4}(\Omega)} \label{dw}%
\end{equation}
where $C_{21}=e^{\left\Vert w_{0}\right\Vert _{L^{\infty}(\Omega)}}\left\Vert
w_{0}\right\Vert _{L^{\infty}(\Omega)}^{\beta}$, $C_{22}=\left\Vert \nabla
w_{0}\right\Vert _{L^{4}(\Omega)}$. Taking into account (\ref{dw}) and the
H\"{o}lder inequality we obtain
\begin{equation}
\left\Vert \nabla w\cdot\nabla v\right\Vert _{L^{2}(\Omega)}\leqslant
C_{22}\left\Vert \nabla v\right\Vert _{L^{4}(\Omega)}+\frac{C_{21}}{2}\frac
{d}{dt}\left[  \int\limits_{\tau}^{t}\left\Vert \nabla v\right\Vert
_{L^{4}(\Omega)}\right]  ^{2}. \label{vw2}%
\end{equation}
In order to estimate the second term of (\ref{vw2}), we obtain from the
H\"{o}lder inequality and (\ref{grv})
\begin{equation}
\left[  \int\limits_{0}^{t}\left\Vert \nabla v\right\Vert _{L^{4}(\Omega
)}\right]  ^{2}\leqslant C_{23}^{2}t^{1/2}\left(  \int\limits_{0}%
^{t}\left\Vert \Delta v\right\Vert _{L^{2}(\Omega)}\right)  \left(
\int\limits_{0}^{t}\int\limits_{\Omega}\left\vert \nabla v\right\vert
^{2}\right)  ^{1/2}\leqslant C_{23}^{2}\sqrt{C_{20}}t\int\limits_{0}%
^{t}\left\Vert \Delta v\right\Vert _{L^{2}(\Omega)}. \label{es2}%
\end{equation}
We estimate now the second term from (\ref{de}). Finally, after integration of
(\ref{vw2}) on $\left[  0,t\right]  $, taking into account (\ref{es2}) and
using the Cauchy inequality and the Gagliardo-Nirenberg inequality in order to
estimate the first term of (\ref{vw2}),
\begin{equation}
\int\limits_{0}^{t}\left\Vert \nabla w\cdot\nabla v\right\Vert _{L^{2}%
(\Omega)}\leqslant\varepsilon\int\limits_{0}^{t}\left\Vert \Delta v\right\Vert
_{L^{2}(\Omega)}+\frac{\left(  C_{22}C_{23}\right)  ^{2}}{4\varepsilon}%
\int\limits_{0}^{t}\left\Vert \nabla v\right\Vert _{L^{2}(\Omega)}%
+\frac{C_{21}C_{23}^{2}\sqrt{C_{20}}}{2}t\int\limits_{0}^{t}\left\Vert \Delta
v\right\Vert _{L^{2}(\Omega)}. \label{t2}%
\end{equation}
We estimate the last term from (\ref{de}) using (\ref{lp})
\begin{equation}
\int\limits_{0}^{t}\left\Vert h(v,w)\right\Vert _{L^{2}(\Omega)}\leqslant
C_{24}t \label{t3}%
\end{equation}
where $C_{24}=\delta\left\Vert v\right\Vert _{L^{2}(\Omega)}+e^{\left\Vert
w_{0}\right\Vert _{L^{\infty}(\Omega)}}\left[  \delta+\left\Vert
w_{0}\right\Vert _{L^{\infty}(\Omega)}^{\beta}\right]  \left\Vert v\right\Vert
_{L^{4}(\Omega)}^{2}$.

Taking into account (\ref{grv}), (\ref{t1}), (\ref{t2}) and (\ref{t3}) we
estimate now $\left\Vert \Delta v\right\Vert _{L^{2}(\Omega)}^{2}$ from
(\ref{de})%
\begin{equation}
\left(  1-\varepsilon-\frac{C_{21}C_{23}^{2}\sqrt{C_{20}}}{2}t\right)
\int\limits_{0}^{t}\left\Vert \Delta v\right\Vert _{L^{2}(\Omega)}%
\leqslant\sqrt{C_{20}}t^{1/2}+t\left[  \frac{\left(  C_{22}C_{23}\right)
^{2}}{4\varepsilon}\sqrt{C_{20}}+C_{24}\right]  . \label{de0}%
\end{equation}
We take $\varepsilon=\frac{1}{4}$ and $t$ sufficiently small such that
\[
1-\varepsilon-\frac{C_{21}C_{23}^{2}\sqrt{C_{20}}}{2}t\geqslant\frac{1}%
{2}\Longrightarrow t\leqslant\frac{1}{2C_{21}C_{23}^{2}\sqrt{C_{20}}}=T_{0}%
\]
which implies from (\ref{de0})%
\begin{equation}
\int\limits_{0}^{t}\left\Vert \Delta v\right\Vert _{L^{2}(\Omega)}%
\leqslant2\sqrt{C_{20}}t^{1/2}+2\left[  \left(  C_{22}C_{23}\right)  ^{2}%
\sqrt{C_{20}}+C_{24}\right]  t=k(t). \label{kt}%
\end{equation}
In this way we have obtained the boundedness for $\int\limits_{0}%
^{t}\left\Vert \Delta v\right\Vert _{L^{2}(\Omega)}$ for all $t\in\left[
0,\min\left\{  T_{0},T\right\}  \right]  $. This bound depends on the initial
data considered in $\tau=0$.

If $T_{0}<T$ we can repeat the procedure taking the initial data in
$\tau=T_{0}$ and in a similar manner we obtain%
\begin{align}
\left(  1-\varepsilon-\frac{C_{21}C_{23}^{2}\sqrt{C_{20}}}{2}\left(
t-T_{0}\right)  \right)   &  \int\limits_{T_{0}}^{t}\left\Vert \Delta
v\right\Vert _{L^{2}(\Omega)}\leqslant\sqrt{C_{20}}\left(  t-T_{0}\right)
^{1/2}+\nonumber\\
&  +\left(  t-T_{0}\right)  \left\{  \frac{C_{23}^{2}\sqrt{C_{20}}%
}{4\varepsilon}\left[  2C_{22}^{2}+C_{21}\int\limits_{0}^{T_{0}}\left\Vert
\Delta v\right\Vert _{L^{2}(\Omega)}\right]  +C_{24}\right\}  . \label{est2}%
\end{align}
Taking $\varepsilon=\frac{1}{4}$ and $T_{0}<t\leqslant2T_{0}$ from
(\ref{est2}) we have
\[
\int\limits_{T_{0}}^{t}\left\Vert \Delta v\right\Vert _{L^{2}(\Omega
)}\leqslant k(t-T_{0})+\frac{1}{T_{0}}\left(  t-T_{0}\right)  \int
\limits_{0}^{T_{0}}\left\Vert \Delta v\right\Vert _{L^{2}(\Omega)}.
\]
The last relation is true for all $t\in\left[  T_{0},\min\left\{
2T_{0},T\right\}  \right]  $. More generally, we obtain
\begin{equation}
\int\limits_{nT_{0}}^{t}\left\Vert \Delta v\right\Vert _{L^{2}(\Omega
)}\leqslant k(t-nT_{0})+\frac{n}{T_{0}}\left(  t-nT_{0}\right)  \int
\limits_{0}^{nT_{0}}\left\Vert \Delta v\right\Vert _{L^{2}(\Omega)}
\label{del}%
\end{equation}
for all $t\in\left[  nT_{0},\min\left\{  \left(  n+1\right)  T_{0},T\right\}
\right]  $, if $n\in\mathbb{N}\backslash\left\{  0\right\}  $ is such that
$nT_{0}<T$.

Let us observe that $t-nT_{0}\leqslant T_{0}$ and the function $k(t)$ given by
(\ref{kt}) is nondecreasing. Thus, the inequality (\ref{del}) becomes%
\[
\int\limits_{nT_{0}}^{t}\left\Vert \Delta v\right\Vert _{L^{2}(\Omega
)}\leqslant k(T_{0})+n\int\limits_{0}^{nT_{0}}\left\Vert \Delta v\right\Vert
_{L^{2}(\Omega)}%
\]
for all $t\in\left[  nT_{0},\min\left\{  \left(  n+1\right)  T_{0},T\right\}
\right]  $.

Finally, for all $t\in\left[  0,\min\left\{  \left(  n+1\right)
T_{0},T\right\}  \right]  $, taking into account (\ref{kt}), we obtain
\begin{align*}
\int\limits_{0}^{t}\left\Vert \Delta v\right\Vert _{L^{2}(\Omega)}  &
\leqslant\left(  n+1\right)  !k(T_{0})\left(  \frac{1}{2!}+...+\frac
{1}{\left(  n+1\right)  !}\right)  +\left(  n+1\right)  !\int\limits_{0}%
^{t}\left\Vert \Delta v\right\Vert _{L^{2}(\Omega)}\leqslant\\
&  \leqslant\left(  n+1\right)  !\left(  1+\frac{1}{2!}+...+\frac{1}{\left(
n+1\right)  !}\right)  k(T_{0})\leqslant e\left(  n+1\right)  !k(T_{0}).
\end{align*}

\end{proof}

\begin{remark}
The last inequality holds for all $0<t<T$, and $n$ is maximal with the
property $nT_{0}\leqslant t$. We emphasize that the bound in terms of $n$ is
equivalent with a bound in terms of $t$, of the same type. Hence we obtain%
\begin{equation}
\int\limits_{0}^{t}\left\Vert \Delta v\right\Vert _{L^{2}(\Omega)}%
\leqslant\Psi_{1}(t) \label{lapf}%
\end{equation}
where $\Psi_{1}$ is a increasing function of the time $t$ having the
properties $\underset{t\searrow0}{\lim}\Psi_{1}(t)=0$, $\underset{t\nearrow
T}{\lim}\Psi_{1}(t)<\infty$ for all $T$ finite.
\end{remark}

Henceforth $\Psi_{i}$, $i=2,3,...$ will stand for a generic function of $t$
having the same properties as the function $\Psi_{1}$.

\begin{lemma}
\label{p4}Let $\Omega\subset\mathbb{R}^{2}$ be a bounded domain. If the
hypothesis $(\mathcal{H})$ is satisfied, $v_{0}\in H^{1}(\Omega)\cap
L^{\infty}(\Omega)$ and $w_{0}\in W^{1,4}(\Omega)$, then we have
\begin{equation}
\left\Vert \nabla v\right\Vert _{L^{1}\left(  0,t;L^{p}(\Omega)\right)
}\leqslant\Psi_{2}(t) \label{estim3}%
\end{equation}
for all $0<t<T$ and $2\leqslant p<\infty$.
\end{lemma}

\begin{proof}
Taking into account the Gagliardo-Nirenberg inequality and the Cauchy
inequality we obtain the following estimate%
\[
\int\limits_{0}^{t}\left\Vert \nabla v\right\Vert _{L^{2^{j}}(\Omega
)}\leqslant\int\limits_{0}^{t}\left\Vert \Delta v\right\Vert _{L^{2}(\Omega
)}+\frac{C_{25}^{2}}{4}\int\limits_{0}^{t}\left\Vert \nabla v\right\Vert
_{L^{2^{j-1}}(\Omega)}%
\]
for $j=2,3,...$. The last inequality implies%
\[
\int\limits_{0}^{t}\left\Vert \nabla v\right\Vert _{L^{2^{j}}(\Omega
)}\leqslant\frac{1-\left(  \frac{C_{25}^{2}}{4}\right)  ^{j-1}}{1-\frac
{C_{25}^{2}}{4}}\int\limits_{0}^{t}\left\Vert \Delta v\right\Vert
_{L^{2}(\Omega)}+\left(  \frac{C_{25}^{2}}{4}\right)  ^{j-1}\int
\limits_{0}^{t}\left\Vert \nabla v\right\Vert _{L^{2}(\Omega)}.
\]
From the last inequality and using (\ref{grv}) and (\ref{lap}) we obtain
(\ref{estim3}) for $p=2^{j}$, $j=2,3,...$. We conclude the lemma from
(\ref{grv}) and taking into account the embeddings of $L^{p}(\Omega)$ spaces.
\end{proof}

\begin{lemma}
\label{p5}Let $\Omega\subset\mathbb{R}^{2}$ be a bounded domain. If the
hypothesis $(\mathcal{H})$ is satisfied, $v_{0}\in H^{1}(\Omega)\cap
L^{\infty}(\Omega)$ and, $w_{0}\in W^{1,p}(\Omega)$, then we have
\[
\left\Vert \nabla w\right\Vert _{L^{p}(\Omega)}\leqslant\Psi_{3}(t)
\]
for all $0<t<T$ and $2<p<\infty$.
\end{lemma}

\begin{proof}
We deduce from the equation (\ref{eq2v})
\[
\nabla w_{t}=-e^{w}w^{\beta}v\nabla w-\beta e^{w}w^{\beta-1}v\nabla
w-e^{w}w^{\beta}\nabla v.
\]
Multiplying this last relation by $\nabla w\left\vert \nabla w\right\vert
^{p-2}$ and after that integrating in $\Omega\times(0,t)$,\ we have%
\[
\left\Vert \nabla w\right\Vert _{L^{p}(\Omega)}\leqslant\left\Vert \nabla
w_{0}\right\Vert _{L^{p}(\Omega)}+e^{\left\Vert w_{0}\right\Vert _{L^{\infty
}(\Omega)}}\left\Vert w_{0}\right\Vert _{L^{\infty}(\Omega)}^{\beta}\left\Vert
\nabla v\right\Vert _{L^{1}\left(  0,t;L^{p}(\Omega)\right)  }.
\]
From the last inequality and Lemma \ref{p4} the statement follows.
\end{proof}

We consider the equation (\ref{eq1v}) together with (\ref{bcv}) and
(\ref{ic1v}) like a linear problem in the variable $v$ in the general form
(\ref{Lv})-(\ref{Lv0}), considering%
\begin{align*}
b_{i}(x,t)  &  =\frac{\partial w}{\partial x_{i}}(x,t),\qquad i=1,2\\
b(x,t)  &  =-w_{t}+\delta(1+w_{t}w^{-\beta})\\
\widetilde{F}(x,t)  &  =\widetilde{G}(x,t)=0.
\end{align*}

Taking $\left(  v_{0},w_{0}\right)  \in W^{2-2/p,p}(\Omega)\times
W^{1,\max\left\{  p,4\right\}  }(\Omega)$, $p\geqslant2$ we observe, taking
into account also the above estimates, that the hypotheses of (\cite[Theorem
9.1, cap. IV]{lady})\ are fulfilled (see also \cite[Theorem II.3]{rascle1}).
This implies that for $p\geqslant2$ we have $v\in W_{p}^{2,1}(\Omega_{T})$.
Moreover, taking into account the embedding results in H\"{o}lder spaces we
obtain $v\in C^{2-4/p,1-2/p}\left(  \overline{\Omega}_{T}\right)  $ and
\begin{equation}
\left\vert v\right\vert _{\Omega_{T}}^{\left(  2-4/p\right)  }\leqslant
\Psi_{3}(t) \label{uc}%
\end{equation}
for all $0<t<T$ and $2<p<\infty$.

\begin{lemma}
Let $\Omega\subset\mathbb{R}^{2}$ be a bounded domain. If the hypothesis
$(\mathcal{H})$ is satisfied, $v_{0}\in H^{1}(\Omega)\cap L^{\infty}(\Omega)$
and $w_{0}\in W^{2,p}(\Omega)$, $p>2$, then $w\in C^{2-4/p,1-2/p}\left(
\overline{\Omega}_{T}\right)  $ and
\begin{equation}
\left\vert w\right\vert _{\Omega_{T}}^{\left(  2-4/p\right)  }\leqslant
\Psi_{4}(t) \label{uw}%
\end{equation}
for all $0<t<T$ and $2<p<\infty$.
\end{lemma}

\begin{proof}
We deduce from the equation (\ref{eq2v})
\[
\Delta w_{t}=-e^{w}\left\{  w^{\beta-1}\left[  w\Delta v+\left(
w+\beta\right)  \left(  2\nabla v\cdot\nabla w+v\Delta w\right)  +v\left(
w+2\beta\right)  \left\vert \nabla w\right\vert ^{2}\right]  -vw^{\beta
-2}\beta(\beta-1)\left\vert \nabla w\right\vert ^{2}\right\}  .
\]
We multiply the last relation by $\left(  \Delta w\right)  ^{p-1}$ with
$p=2^{j}$, $j=1,2,...$. Integrating after that in $\Omega$ and applying
Young's inequality we obtain%
\begin{equation}
\left\Vert \Delta w\right\Vert _{L^{p}(\Omega)}^{p}\leqslant
e^{4(p-1)\varepsilon^{\left(  p/(p-1)\right)  t}}\left[  \left\Vert \Delta
w_{0}\right\Vert _{L^{p}(\Omega)}^{p}+\int\limits_{0}^{t}M(s)ds\right]
\label{delta}%
\end{equation}
where%
\begin{align*}
&  \int\limits_{0}^{t}M(s)ds=\varepsilon^{-p}e^{p\left\Vert w_{0}\right\Vert
_{L^{\infty}(\Omega)}}\left[  \left\Vert w_{0}\right\Vert _{L^{\infty}%
(\Omega)}^{p\beta}\left(  \left\Vert \Delta v\right\Vert _{L^{1}\left(
0,t;L^{p}(\Omega)\right)  }^{p}+2\int\limits_{0}^{t}\left\Vert \nabla
v\cdot\nabla w\right\Vert _{L^{p}(\Omega)}^{p}\right)  \right.  +\\
&  +2\beta\left\Vert w_{0}\right\Vert _{L^{\infty}(\Omega)}^{p\left(
\beta-1\right)  }\left(  \int\limits_{0}^{t}\left\Vert \nabla v\cdot\nabla
w\right\Vert _{L^{p}(\Omega)}^{p}+\left\Vert v\right\Vert _{L^{\infty
}(0,t;L^{\infty}\left(  \Omega\right)  )}^{p}\int\limits_{0}^{t}\left\Vert
\nabla w\right\Vert _{L^{2p}(\Omega)}^{2p}\right)  +\\
&  \left.  +\left\Vert v\right\Vert _{L^{\infty}(0,t;L^{\infty}\left(
\Omega\right)  )}^{p}\left\Vert w_{0}\right\Vert _{L^{\infty}(\Omega
)}^{p(\beta-2)}\left(  \left\Vert w_{0}\right\Vert _{L^{\infty}(\Omega)}%
^{2}+\beta\left\vert \beta-1\right\vert \right)  ^{p}\int\limits_{0}%
^{t}\left\Vert \nabla w\right\Vert _{L^{2p}(\Omega)}^{2p}\right]  .
\end{align*}
Next observe that Lemma \ref{p4}, Lemma \ref{p5} and (\ref{uc}) allow us to
estimate the integral on the right-hand side of (\ref{delta}) and to obtain%
\[
\left\Vert \Delta w\right\Vert _{L^{p}(\Omega)}^{p}\leqslant
e^{4(p-1)\varepsilon^{\left(  p/(p-1)\right)  t}}\left[  \left\Vert \Delta
w_{0}\right\Vert _{L^{p}(\Omega)}^{p}+\Psi_{4}(t)\right]  .
\]
In a similar way, we obtain $D_{x}^{s}w\in L^{p}(\Omega_{T}),$ $\left\vert
s\right\vert =2$. Taking into account Lemma \ref{p5} and the embedding results
in H\"{o}lder space we conclude the proof.
\end{proof}

\begin{lemma}
Let $\Omega\subset\mathbb{R}^{2}$ be a bounded domain. If the hypothesis
$(\mathcal{H})$ is satisfied and $\left(  u_{0},w_{0}\right)  \in\left(
W^{2,p}(\Omega)\right)  ^{2}$, $p>2$, then $u\in C^{2-4/p,1-2/p}\left(
\overline{\Omega}_{T}\right)  $ and
\[
\left\vert u\right\vert _{\Omega_{T}}^{(2-4/q)}\leqslant\Psi_{5}(t)
\]
for all $T$ finite, $0<t<T$.
\end{lemma}

\begin{proof}
The conclusion of the Lemma follows from (\ref{uc}), (\ref{uw}) and Lemma
\ref{R}.
\end{proof}

To achieve the proof of Theorem \ref{globalex} we use the next Lemma whose
proof is similar to the proof of Lemmas IV.2 and IV.3 \ in \cite{rascle1} (see
also Lemma 2 in \cite{rascle2}).

\begin{lemma}
\label{Rasclelem} (i) Suppose that $\left\Vert u\right\Vert _{\Omega_{T}%
}^{(m)}\leqslant\Psi(t)$, $m>1$, $m$ not integer, for all $0\leqslant t<T$.
Then we have
\[
\left\vert w^{\beta}\Delta U\right\vert _{\Omega_{t}}^{(\alpha)}+\left\vert
w^{\beta}\nabla U\right\vert _{\Omega_{t}}^{(\alpha)}+\left\vert
w^{-1}\left\vert \nabla w\right\vert ^{2}\right\vert _{\Omega_{t}}^{(\alpha
)}\leqslant\Psi(t)
\]
for all $0\leqslant t<T$, where $\alpha=\min\left\{  l+2,m-1\right\}  $.

(ii) Let $\alpha>0$ not integer. If
\[
\left\vert w^{\beta}\Delta U\right\vert _{\Omega_{t}}^{(\alpha)}+\left\vert
w^{\beta}\nabla U\right\vert _{\Omega_{t}}^{(\alpha)}+\left\vert
w^{-1}\left\vert \nabla w\right\vert ^{2}\right\vert _{\Omega_{t}}^{(\alpha
)}\leqslant\Psi(t)
\]
for all $0\leqslant t<\tau$, then%
\[
\left\vert u\right\vert _{\Omega_{t}}^{(\eta+2)}\leqslant\Psi(t)
\]
where $\eta=\min\left\{  \alpha,l,m\right\}  $.
\end{lemma}

In such a way the regularity of the solution $u$ is successively ameliorated
until reaching the desired bound of $\left\vert u\right\vert _{\Omega_{t}%
}^{(l+2)}$.

\section{Asymptotic behavior of global solutions}

\label{asymptotic}

\setcounter{equation}{0} \setcounter{figure}{0}

\subsection{Steady states}

In this Section we are going to study the asymptotic behavior of the smooth
solution of the problem (\ref{eq1})-(\ref{ic2}). We shall begin by analyzing
the steady states of the system (\ref{eq1})- (\ref{eq2}) with homogeneous
Neumann boundary condition (\ref{bc}). So, we consider the following
stationary problem:%
\begin{align}
&  0=\Delta u-\nabla\cdot(u\nabla w)+\delta u(1-u) & x  &  \in\Omega
\label{stationary1}\\
&  0=w^{\beta}u & x  &  \in\Omega\label{stationary2}\\
&  \frac{\partial u}{\partial\eta}-u\frac{\partial w}{\partial\eta}=0 & x  &
\in\partial\Omega. \label{stationary3}%
\end{align}

\begin{lemma}
\label{intFG}Let $\Omega\subset\mathbb{R}^{N}$, $N\geqslant1$ be a domain. Let
$u,w\in C^{1}(\Omega)$ be two functions satisfying $uw=0$, for all $x\in
\Omega$. Then we have $\nabla u\cdot\nabla w=0$ for all $x\in\Omega$.
\end{lemma}

\begin{proof}
We consider the closed sets%
\[
F=u^{-1}(0),\quad G=w^{-1}(0).
\]
The fact that $uw=0$ implies that $F\cup G=\Omega.$ As $F$ and $G$ are closed
it is straightforward to show that
\begin{equation}
\overline{\inter F}\cup\overline{\inter G}=\Omega. \label{int}%
\end{equation}
As the functions $u$ and $w$ belong to $C^{1}(\Omega)$, the sets $\left(
\nabla u\right)  ^{-1}\left(  0\right)  $, $\left(  \nabla w\right)
^{-1}\left(  0\right)  $ are closed. Moreover, $\inter F\subset\left(  \nabla
u\right)  ^{-1}\left(  0\right)  $, $\inter G\subset\left(  \nabla w\right)
^{-1}\left(  0\right)  $, which imply, taking into account (\ref{int})%
\[
\Omega\subset\left(  \nabla u\right)  ^{-1}\left(  0\right)  \cup\left(
\nabla w\right)  ^{-1}\left(  0\right)
\]
and the proof is complete.
\end{proof}

\begin{proposition}
\label{th_steady_states}Let $\Omega\subset\mathbb{R}^{N}$, $N\geqslant1$ be an
open set. If $(u,w)\in\left(  C^{2}(\Omega)\cap C^{1}(\overline{\Omega
})\right)  ^{2}$ is a classical solution to (\ref{stationary1}%
)-(\ref{stationary3}) then%
\[
(u,w)=(0,\widetilde{w})\,\quad\text{or\quad}(u,w)=(k,0)\,
\]
where $\widetilde{w}\in C^{2}(\Omega)\cap C^{1}(\overline{\Omega})$ and $k$ is
a constant if $\delta=0$ and $k=1$ if $\delta>0$.
\end{proposition}

\begin{proof}
If $\delta=0$, we multiply (\ref{stationary1}) by $u$ and integrate over
$\Omega$. We obtain%
\[
0=-\int\limits_{\Omega}\left\vert \nabla u\right\vert ^{2}+\int\limits_{\Omega
}u\nabla w\cdot\nabla u.
\]
The last equality and Lemma \ref{intFG} imply that $u$ is a constant. Taking
also into account (\ref{stationary2})-(\ref{stationary3}), the conclusion of
the theorem follows.

We now turn to the case $\delta>0$. Multiplying (\ref{stationary1}) by $u-1$
and integrating over $\Omega$ we obtain%
\[
0=-\int\limits_{\Omega}\left\vert \nabla u\right\vert ^{2}+\int\limits_{\Omega
}u\nabla w\cdot\nabla u-\int\limits_{\Omega}\delta u(u-1)^{2}.
\]
From Lemma \ref{intFG} we have%
\[
\int\limits_{\Omega}\left\vert \nabla u\right\vert ^{2}=-\int\limits_{\Omega
}\delta u(u-1)^{2}<0.
\]
We conclude the proof using the same arguments as above.
\end{proof}

\ 

In the remaining of this paper we shall place ourselves in the hypotheses of
Theorem \ref{globalex}. Then the system (\ref{eq1})-(\ref{ic2}) has a global
in time classical H\"{o}lder continuous solution. We emphasize that the
hypothesis $(\mathcal{H})$ is also fulfilled.

\ \ 

\begin{lemma}
\label{pos}If there exists a positive constant $\gamma>0$ such that
$u_{0}(x)\geqslant\gamma$ for all $x\in\Omega$, then every global solution $u$
of (\ref{eq1})-(\ref{ic2}) satisfies $u(x,t)\geqslant\min\left\{
1,\gamma\right\}  e^{-\left\Vert w_{0}\right\Vert _{L^{\infty}(\Omega)}}$ for
all $x\in\Omega$, $t>0$.
\end{lemma}

\begin{proof}
Let $\alpha$ be a positive constant to be chosen later. By multiplying the
equation (\ref{eq1v}) by $e^{w}(v-\alpha)_{-}=e^{w}\max\left\{  \alpha
-v,0\right\}  $ and integrating over $\Omega$ we get%

\begin{align}
\frac{1}{2}\frac{d}{dt}\int\limits_{\Omega}e^{w}(v-\alpha)_{-}^{2}  &
=-\int\limits_{\Omega}e^{w}\left\vert \nabla(v-\alpha)_{-}\right\vert
^{2}-\int\limits_{\Omega}e^{2w}v^{2}w^{\beta}(v-\alpha)_{-}-\nonumber\\
&  -\frac{1}{2}\int\limits_{\Omega}e^{2w}vw^{\beta}(v-\alpha)_{-}^{2}%
-\delta\int\limits_{\Omega}e^{w}v(1-ve^{w})(v-\alpha)_{-}. \label{bound}%
\end{align}
If $\delta=0$, let us observe that the right-hand side of (\ref{bound}) is
nonpositive. If $\delta>0$, we choose $\alpha$ small enough such that
$0<\alpha\leqslant e^{-\left\Vert w_{0}\right\Vert _{L^{\infty}(\Omega)}}$.
Then the last term in (\ref{bound}) is also nonpositive.

From the above considerations, we get%
\begin{equation}
\int\limits_{\Omega}e^{w}(v-\alpha)_{-}^{2}\leqslant\int\limits_{\Omega
}e^{w_{0}}(v_{0}-\alpha)_{-}^{2}. \label{ine}%
\end{equation}
We consider first the case when $\gamma<1$. We choose $\alpha=\gamma
e^{-\left\Vert w_{0}\right\Vert _{L^{\infty}(\Omega)}}<e^{-\left\Vert
w_{0}\right\Vert _{L^{\infty}(\Omega)}}$. Because $u_{0}\geqslant\gamma>0$, we
obtain $v_{0}=u_{0}e^{-w_{0}}\geqslant\gamma e^{-\left\Vert w_{0}\right\Vert
_{L^{\infty}(\Omega)}}=\alpha$. From (\ref{ine}) we obtain that $u\geqslant
\gamma e^{-\left\Vert w_{0}\right\Vert _{L^{\infty}(\Omega)}}$.

Now let $\gamma\geqslant1$. We choose $\alpha=e^{-\left\Vert w_{0}\right\Vert
_{L^{\infty}(\Omega)}}$. Because $u_{0}\geqslant\gamma>1$, we obtain
$v_{0}=u_{0}e^{-w_{0}}\geqslant e^{-\left\Vert w_{0}\right\Vert _{L^{\infty
}(\Omega)}}=\alpha$. From (\ref{ine}) we obtain that $u\geqslant
e^{-\left\Vert w_{0}\right\Vert _{L^{\infty}(\Omega)}}$.
\end{proof}

\subsection{Exponential convergence}

In this subsection we consider $\beta=1.$

\begin{lemma}
\label{estim copy(1)}If there exists a positive constant $\gamma>0$ such that
$u_{0}(x)\geqslant\gamma$ for all $x\in\Omega$, then%
\begin{equation}
\int\limits_{\Omega}\left\vert \nabla w\right\vert ^{2}\leqslant\left(
C_{27}+C_{28}t\right)  e^{-2\lambda t} \label{w1}%
\end{equation}
where $C_{27}$, $C_{28}$ are positive constants independent on $t$ and
$\lambda=\min\left\{  1,\gamma\right\}  e^{-\left\Vert w_{0}\right\Vert
_{L^{\infty}(\Omega)}}>0$.
\end{lemma}

\begin{proof}
From (\ref{eq2}) we obtain%
\[
\left\vert \nabla w\right\vert ^{2}\leqslant2e^{-2\int\limits_{0}^{t}%
u}\left\vert \nabla w_{0}\right\vert ^{2}+2te^{-2\int\limits_{0}^{t}%
u}\left\vert w_{0}\right\vert ^{2}\int\limits_{0}^{t}\left\vert \nabla
u\right\vert ^{2}.
\]
Taking into account Lemma \ref{pos}\ we know that $u(x,t)\geqslant\lambda>0$
for all $x\in\Omega$, $t>0$. We have from the previous inequality
\begin{equation}
\int\limits_{\Omega}\left\vert \nabla w\right\vert ^{2}\leqslant2e^{-2\lambda
t}\int\limits_{\Omega}\left\vert \nabla w_{0}\right\vert ^{2}+2te^{-2\lambda
t}\left\Vert w_{0}\right\Vert _{L^{\infty}(\Omega)}^{2}\int\limits_{0}^{t}%
\int\limits_{\Omega}\left\vert \nabla u\right\vert ^{2}. \label{iexp}%
\end{equation}
Taking into account the hypothesis $(\mathcal{H})$, the estimates
(\ref{Lyap1}), (\ref{F_delta}) and because every term of the functional
$D(u,w)$ given by (\ref{Dly}) is positive, we obtain that the last term in
(\ref{iexp}) is bounded. More precisely
\begin{equation}
\int\limits_{0}^{t}\int\limits_{\Omega}\left\vert \nabla u\right\vert
^{2}dxds\leqslant\left\Vert u\right\Vert _{L^{\infty}(0,t;L^{\infty}(\Omega
))}\int\limits_{0}^{t}D(u,w)\leqslant C_{26}\left\Vert u\right\Vert
_{L^{\infty}(0,t;L^{\infty}(\Omega))} \label{du}%
\end{equation}
where $C_{26}=F(u_{0},w_{0})+\left\vert \Omega\right\vert $. Finally, from
(\ref{iexp}) and (\ref{du}) we obtain
\[
\int\limits_{\Omega}\left\vert \nabla w\right\vert ^{2}\leqslant2e^{-2\lambda
t}\left[  \left\Vert \nabla w_{0}\right\Vert _{L^{2}(\Omega)}^{2}%
+C_{26}t\left\Vert w_{0}\right\Vert _{L^{\infty}(\Omega)}^{2}\left\Vert
u\right\Vert _{L^{\infty}(0,t;L^{\infty}(\Omega))}\right]  =\left(
C_{27}+C_{28}t\right)  e^{-2\lambda t}%
\]
and we conclude the proof.
\end{proof}

\begin{proposition}
\label{p8}If there exists a positive constant $\gamma>0$ such that
$u_{0}(x)\geqslant\gamma$ for all $x\in\Omega$, then%
\begin{align*}
&  \left\Vert u(\cdot,t)-\overline{u}\right\Vert _{L^{2}(\Omega)}%
\leqslant\left\Vert u\right\Vert _{L^{\infty}(0,t;L^{\infty}(\Omega))}%
^{2}\left(  C_{27}+\frac{C_{28}}{2}t\right)  te^{-C_{30}t}\text{,} & \text{if
}\delta &  =0\\
&  \left\Vert u(\cdot,t)-1\right\Vert _{L^{2}(\Omega)}\leqslant\left[
\left\Vert u_{0}-1\right\Vert _{L^{2}(\Omega)}^{2}+\left\Vert u\right\Vert
_{L^{\infty}(0,t;L^{\infty}(\Omega))}^{2}\left(  C_{27}+\frac{C_{28}}%
{2}t\right)  t\right]  e^{-2\lambda\min\left\{  1,\delta\right\}  t}\text{,} &
\text{if }\delta &  >0\\
&  \left\Vert w(\cdot,t)\right\Vert _{L^{\infty}(\Omega)}\leqslant\left\Vert
w_{0}\right\Vert _{L^{\infty}(\Omega)}\,e^{-\lambda t} &  &
\end{align*}
where $\overline{u}=\frac{1}{\left\vert \Omega\right\vert }\int\limits_{\Omega
}u_{0}$, $\lambda=\min\left\{  1,\gamma\right\}  e^{-\left\Vert w_{0}%
\right\Vert _{L^{\infty}(\Omega)}}>0$ and $C_{27}$, $C_{28}$, $C_{30}$ are
positive constants independent on $t$.
\end{proposition}

\begin{proof}
Let $\sigma$ be a positive constant to be chosen later. We multiply the
equation (\ref{eq1}) by $(u-\sigma)$ and integrate over $\Omega$
\begin{equation}
\frac{d}{dt}\int\limits_{\Omega}\left(  u-\sigma\right)  ^{2}\leqslant
-\int\limits_{\Omega}\left\vert \nabla u\right\vert ^{2}+\left\Vert
u\right\Vert _{L^{\infty}(0,t;L^{\infty}(\Omega))}^{2}\int\limits_{\Omega
}\left\vert \nabla w\right\vert ^{2}-2\delta\int\limits_{\Omega}%
u(u-1)(u-\sigma). \label{sigma}%
\end{equation}
First we consider the case $\delta=0$ and $\sigma=\overline{u}=\frac
{1}{\left\vert \Omega\right\vert }\int\limits_{\Omega}u_{0}(x)$. Applying the
Poincar\'{e} inequality in (\ref{sigma}) we obtain%
\begin{equation}
\frac{d}{dt}\int\limits_{\Omega}\left(  u-\overline{u}\right)  ^{2}+C_{29}%
\int\limits_{\Omega}\left(  u-\overline{u}\right)  ^{2}\leqslant\left\Vert
u\right\Vert _{L^{\infty}(0,t;L^{\infty}(\Omega))}^{2}\int\limits_{\Omega
}\left\vert \nabla w\right\vert ^{2}. \label{de1}%
\end{equation}
Applying the Gronwall inequality in the last estimate and taking into account
(\ref{w1})\ we have%
\[
\int\limits_{\Omega}\left(  u-\overline{u}\right)  ^{2}\leqslant\left\Vert
u\right\Vert _{L^{\infty}(0,t;L^{\infty}(\Omega))}^{2}\left(  C_{27}%
+\frac{C_{28}}{2}t\right)  te^{-C_{30}t}%
\]
where $C_{30}=\min\left\{  2\lambda,C_{29}\right\}  $.

Let now $\delta\neq0$ and $\sigma=1$. Using Lemma \ref{pos} we obtain from
(\ref{sigma})
\begin{equation}
\frac{d}{dt}\int\limits_{\Omega}\left(  u-1\right)  ^{2}+2\lambda\delta
\int\limits_{\Omega}(u-1)^{2}\leqslant\left\Vert u\right\Vert _{L^{\infty
}(0,t;L^{\infty}(\Omega))}^{2}\int\limits_{\Omega}\left\vert \nabla
w\right\vert ^{2}. \label{de2}%
\end{equation}
Applying Gronwall's inequality and (\ref{w1}) it follows that
\[
\int\limits_{\Omega}\left(  u-1\right)  ^{2}\leqslant\left[  \left\Vert
u_{0}-1\right\Vert _{L^{2}(\Omega)}^{2}+\left\Vert u\right\Vert _{L^{\infty
}(0,t;L^{\infty}(\Omega))}^{2}\left(  C_{27}+\frac{C_{28}}{2}t\right)
t\right]  e^{-2\lambda\min\left\{  1,\delta\right\}  t}.
\]
Finally, from (\ref{wexpr}) we obtain
\[
w=w_{0}e^{-\int\limits_{0}^{t}u}\leqslant w_{0}e^{-\lambda t}%
\]
and we conclude the proof.
\end{proof}

\subsection{Polynomial convergence}

In this subsection we consider $\beta>1.$

\begin{lemma}
\label{estim}If there exists a positive constant $\gamma>0$ such that
$u_{0}(x)\geqslant\gamma$ for all $x\in\Omega$, then
\begin{equation}
\int\limits_{0}^{t}\left(  \left\Vert w_{0}\right\Vert _{L^{\infty}\left(
\Omega\right)  }^{1-\beta}+\lambda\left(  \beta-1\right)  s\right)  ^{\frac
{1}{\beta-1}}\int\limits_{\Omega}\left\vert \nabla w\right\vert ^{2}%
dxds\leqslant C_{31}\label{w2}%
\end{equation}
where $C_{31}$ is a positive constant independent on $t$ and $\lambda
=\min\left\{  1,\gamma\right\}  e^{-\left\Vert w_{0}\right\Vert _{L^{\infty
}(\Omega)}}>0$.
\end{lemma}

\begin{proof}
Taking into account the hypothesis $(\mathcal{H})$ and the estimates
(\ref{Lyap1}) and (\ref{F_delta}) we have
\begin{equation}
\int\limits_{0}^{t}\int\limits_{\Omega}uw^{-1}\left\vert \nabla w\right\vert
^{2}dxds\leqslant\int\limits_{0}^{t}D(u,w)\leqslant C_{26}. \label{intD}%
\end{equation}
From Lemma \ref{pos}\ we know that $u(x,t)\geqslant\lambda>0$ for all
$x\in\Omega$, $t>0$. Taking into account (\ref{wexpr}) we obtain%
\begin{equation}
uw^{-1}\geqslant\lambda\left(  w_{0}^{1-\beta}+\lambda\left(  \beta-1\right)
t\right)  ^{\frac{1}{\beta-1}}. \label{uw1}%
\end{equation}
The inequalities (\ref{intD}) and (\ref{uw1}) imply%
\[
\int\limits_{0}^{t}\left(  \left\Vert w_{0}\right\Vert _{L^{\infty}\left(
\Omega\right)  }^{1-\beta}+\lambda\left(  \beta-1\right)  s\right)  ^{\frac
{1}{\beta-1}}\int\limits_{\Omega}\left\vert \nabla w\right\vert ^{2}%
dxds\leqslant\lambda^{-1}\int\limits_{0}^{t}\int\limits_{\Omega}%
uw^{-1}\left\vert \nabla w\right\vert ^{2}dxds\leqslant\lambda^{-1}C_{26}%
\]
and we conclude the proof with $C_{31}=\lambda^{-1}C_{26}$.
\end{proof}

\begin{proposition}
If there exists a positive constant $\gamma>0$ such that $u_{0}(x)\geqslant
\gamma$ for all $x\in\Omega$, then%
\begin{align*}
&  \left\Vert u(\cdot,t)-\overline{u}\right\Vert _{L^{2}(\Omega)}\leqslant
C_{32}\left(  \left\Vert w_{0}\right\Vert _{L^{\infty}\left(  \Omega\right)
}^{1-\beta}+\lambda\left(  \beta-1\right)  t\right)  ^{-\frac{1}{\beta-1}%
}\text{,} & \text{if }\delta &  =0\\
&  \left\Vert u(\cdot,t)-1\right\Vert _{L^{2}(\Omega)}\leqslant C_{33}\left(
\left\Vert w_{0}\right\Vert _{L^{\infty}\left(  \Omega\right)  }^{1-\beta
}+\lambda\left(  \beta-1\right)  t\right)  ^{-\frac{1}{\beta-1}}\text{,} &
\text{if }\delta &  >0\\
&  \left\Vert w(\cdot,t)\right\Vert _{L^{\infty}(\Omega)}\leqslant\left[
\left\Vert w_{0}\right\Vert _{L^{\infty}\left(  \Omega\right)  }^{1-\beta
}+\lambda\left(  \beta-1\right)  t\right]  ^{-\frac{1}{\beta-1}} &  &
\end{align*}
where $\overline{u}=\frac{1}{\left\vert \Omega\right\vert }\int\limits_{\Omega
}u_{0}$, $\lambda=\min\left\{  1,\gamma\right\}  e^{-\left\Vert w_{0}%
\right\Vert _{L^{\infty}(\Omega)}}>0$ and $C_{32}$, $C_{33}$ are positive
constants independent on $t$.
\end{proposition}

\begin{proof}
First we consider the case $\delta=0$.

Let us observe that for $t>t_{01}=\max\left\{  0,\frac{1}{\lambda\left(
\beta-1\right)  }\left(  \frac{\lambda}{C_{29}}-\left\Vert w_{0}\right\Vert
_{L^{\infty}\left(  \Omega\right)  }^{1-\beta}\right)  \right\}  $ we have
\begin{align}
&  \frac{d}{dt}\left[  \left(  \left\Vert w_{0}\right\Vert _{L^{\infty}\left(
\Omega\right)  }^{1-\beta}+\lambda\left(  \beta-1\right)  t\right)  ^{\frac
{1}{\beta-1}}\int\limits_{\Omega}\left(  u-\overline{u}\right)  ^{2}dx\right]
=\nonumber\\
&  =\left(  \left\Vert w_{0}\right\Vert _{L^{\infty}\left(  \Omega\right)
}^{1-\beta}+\lambda\left(  \beta-1\right)  t\right)  ^{\frac{1}{\beta-1}%
}\left[  \frac{d}{dt}\int\limits_{\Omega}\left(  u-\overline{u}\right)
^{2}dx+\lambda\left(  \left\Vert w_{0}\right\Vert _{L^{\infty}\left(
\Omega\right)  }^{1-\beta}+\lambda\left(  \beta-1\right)  t\right)  ^{-1}%
\int\limits_{\Omega}\left(  u-\overline{u}\right)  ^{2}dx\right]
\leqslant\nonumber\\
&  \leqslant\left(  \left\Vert w_{0}\right\Vert _{L^{\infty}\left(
\Omega\right)  }^{1-\beta}+\lambda\left(  \beta-1\right)  t\right)  ^{\frac
{1}{\beta-1}}\left[  \frac{d}{dt}\int\limits_{\Omega}\left(  u-\overline
{u}\right)  ^{2}+C_{29}\int\limits_{\Omega}(u-\overline{u})^{2}\right]  .
\label{esm}%
\end{align}
We multiply (\ref{de1}) by $\left(  \left\Vert w_{0}\right\Vert _{L^{\infty
}\left(  \Omega\right)  }^{1-\beta}+\lambda\left(  \beta-1\right)  t\right)
^{\frac{1}{\beta-1}}$ and then we integrate between $t_{01}$ and an arbitrary
$t>t_{01}$. Taking into account (\ref{w2}) and (\ref{esm}) we have%
\[
\int\limits_{\Omega}\left(  u-\overline{u}\right)  ^{2}dx\leqslant
C_{33}\left(  \left\Vert w_{0}\right\Vert _{L^{\infty}\left(  \Omega\right)
}^{1-\beta}+\lambda\left(  \beta-1\right)  t\right)  ^{-\frac{1}{\beta-1}}%
\]
where $C_{32}=\frac{1}{2\varepsilon}C_{31}\left\Vert u\right\Vert _{L^{\infty
}(0,t;L^{\infty}(\Omega))}^{2}+\left(  \left\Vert w_{0}\right\Vert
_{L^{\infty}\left(  \Omega\right)  }^{1-\beta}+\lambda\left(  \beta-1\right)
t_{01}\right)  ^{\frac{1}{\beta-1}}\int\limits_{\Omega}\left(  u(x,t_{01}%
)-\overline{u}\right)  ^{2}dx$.

If $\delta\neq0$ we have for $t>t_{02}=\max\left\{  0,\frac{1}{\lambda\left(
\beta-1\right)  }\left(  \frac{1}{2\delta}-\left\Vert w_{0}\right\Vert
_{L^{\infty}\left(  \Omega\right)  }^{1-\beta}\right)  \right\}  $
\begin{align}
&  \frac{d}{dt}\left[  \left(  \left\Vert w_{0}\right\Vert _{L^{\infty}\left(
\Omega\right)  }^{1-\beta}+\lambda\left(  \beta-1\right)  s\right)  ^{\frac
{1}{\beta-1}}\int\limits_{\Omega}\left(  u-1\right)  ^{2}dx\right]
\leqslant\nonumber\\
&  \leqslant\left(  \left\Vert w_{0}\right\Vert _{L^{\infty}\left(
\Omega\right)  }^{1-\beta}+\lambda\left(  \beta-1\right)  s\right)  ^{\frac
{1}{\beta-1}}\left[  \frac{d}{dt}\int\limits_{\Omega}\left(  u-1\right)
^{2}+2\lambda\delta\int\limits_{\Omega}(u-1)^{2}\right]  . \label{est}%
\end{align}
Multiplying (\ref{de2}) by $\left(  \left\Vert w_{0}\right\Vert _{L^{\infty
}\left(  \Omega\right)  }^{1-\beta}+\lambda\left(  \beta-1\right)  t\right)
^{\frac{1}{\beta-1}}$, integrating between $t_{02}$ and an arbitrary
$t>t_{02}$ and taking into account the last inequality and (\ref{w2}), we
obtain%
\[
\int\limits_{\Omega}\left(  u-1\right)  ^{2}dx\leqslant C_{33}\left(
\left\Vert w_{0}\right\Vert _{L^{\infty}\left(  \Omega\right)  }^{1-\beta
}+\lambda\left(  \beta-1\right)  t\right)  ^{-\frac{1}{\beta-1}}%
\]
where $C_{33}=\frac{1}{2\varepsilon}C_{31}\left\Vert u\right\Vert _{L^{\infty
}(0,t;L^{\infty}(\Omega))}^{2}+\left(  \left\Vert w_{0}\right\Vert
_{L^{\infty}\left(  \Omega\right)  }^{1-\beta}+\lambda\left(  \beta-1\right)
t_{02}\right)  ^{\frac{1}{\beta-1}}\int\limits_{\Omega}\left(  u\left(
x,t_{02}\right)  -1\right)  ^{2}dx$.

From (\ref{wexpr}) we obtain%
\[
w=\left[  w_{0}^{1-\beta}+\left(  \beta-1\right)  \int\limits_{0}^{t}u\right]
^{\frac{1}{1-\beta}}\leqslant\left[  w_{0}^{1-\beta}+\lambda\left(
\beta-1\right)  t\right]  ^{\frac{1}{1-\beta}}%
\]
and we conclude the proof.
\end{proof}

\section{Acknowledgements}

This work was partially supported by the RTN \textquotedblright Modeling,
Mathematical Methods and Computer Simulation of Tumour Growth and
Therapy\textquotedblright\ (MRTN-CT-2004-503661). The first author was also
partially supported by projects DGES (Spain) Grant MTM2007-61755 and CEEX
Grant 05-D11-36/05.10.05.

\end{document}